\documentclass[10pt, a4paper]{amsart}
\usepackage{amscd,amsmath,amssymb,amsfonts,amsthm,ascmac}
\usepackage{enumerate,mathrsfs,stmaryrd,latexsym, comment, mathtools} 
\usepackage[all]{xy}
\usepackage[top=27truemm,bottom=21truemm,left=26truemm,right=26truemm]{geometry}
\usepackage{graphicx}
\def\r{\mathbb{R}}
\def\c{\mathbb{C}}
\def\q{\mathbb{Q}}
\def\z{\mathbb{Z}}

\newtheorem{thm}{Theorem}[section]
\newtheorem{defi}[thm]{Definition}
\newtheorem{rem}[thm]{Remark}
\newtheorem{prop}[thm]{Proposition}
\newtheorem{ex}[thm]{Example}
\newtheorem{cor}[thm]{Corollary}
\newtheorem{lem}[thm]{Lemma}

\newtheorem*{ack}{Acknowledgements}
\newtheorem*{theor}{Theorem}

 \makeatletter
    
    \@addtoreset{equation}{section}
  \makeatother%数式番号にsection number
\title[FOLDING PROCEDURE FOR NEWTON-OKOUNKOV POLYTOPES OF SCHUBERT VARIETIES]{\fontsize{11pt}{11pt}\selectfont FOLDING PROCEDURE FOR NEWTON-OKOUNKOV POLYTOPES OF SCHUBERT VARIETIES}
\date{\today}
\author[N. Fujita]{\fontsize{10pt}{10pt}\selectfont Naoki Fujita}
\begin{document}
\address{Department of Mathematics, Tokyo Institute of Technology, 2-12-1 Oh-okayama, Meguro-ku, Tokyo 152-8551, Japan}
\email{fujita.n.ac@m.titech.ac.jp}
\thanks{The work was partially supported by Grant-in-Aid for JSPS Fellows (No.\ 16J00420)}
\begin{abstract}
The theory of Newton-Okounkov polytopes is a generalization of that of Newton polytopes for toric varieties, and it gives a systematic method of constructing toric degenerations of a projective variety. In the case of Schubert varieties, their Newton-Okounkov polytopes are deeply connected with representation theory. Indeed, Littelmann's string polytopes and Nakashima-Zelevinsky's polyhedral realizations are obtained as Newton-Okounkov polytopes of Schubert varieties. In this paper, we apply the folding procedure to a Newton-Okounkov polytope of a Schubert variety, which relates Newton-Okounkov polytopes of Schubert varieties of different types. As an application of this result, we obtain a new interpretation of Kashiwara's similarity of crystal bases.
\end{abstract}
\maketitle
\setcounter{tocdepth}{2}
\tableofcontents
\section{Introduction}
This paper is devoted to the study of the folding procedure for a Newton-Okounkov polytope of a Schubert variety. The theory of Newton-Okounkov polytopes was introduced by Okounkov \cite{Oko1, Oko2}, and afterward developed independently by Kaveh-Khovanskii \cite{KK1} and by Lazarsfeld-Mustata \cite{LM}. It is a generalization of the theory of Newton polytopes for toric varieties to arbitrary projective varieties, and it gives a systematic method of constructing toric degenerations by \cite[Theorem 1]{And} (see also \cite{HK}). In the case of Schubert varieties, their Newton-Okounkov polytopes include some representation-theoretic polytopes such as Littelmann's string polytopes \cite{Kav}, Nakashima-Zelevinsky's polyhedral realizations \cite{FN}, and Feigin-Fourier-Littelmann-Vinberg's polytopes \cite{FeFL, Kir}; in addition, Lusztig's parametrization of the canonical basis also appears in the theory of Newton-Okounkov polytopes (see \cite{FaFL}). In this paper, we study Littelmann's string polytopes and Nakashima-Zelevinsky's polyhedral realizations, and obtain relations among these polytopes for Schubert varieties of different types.

To be more precise, let $\mathfrak{g}$ be a simply-laced simple Lie algebra, $\mathfrak{t} \subset \mathfrak{g}$ a Cartan subalgebra, $P_+ \subset \mathfrak{t}^\ast$ the set of dominant weights for $\mathfrak{g}$, and $\omega \colon I \rightarrow I$ a Dynkin diagram automorphism, where $I$ is an index set for the vertices of the Dynkin diagram. In this paper, for technical reasons, we always assume that any two vertices of the Dynkin diagram in the same $\omega$-orbit are not joined. Such an $\omega$ induces a Lie algebra automorphism $\omega \colon \mathfrak{g} \xrightarrow{\sim} \mathfrak{g}$, which preserves the Cartan subalgebra $\mathfrak{t}$. We know that the fixed point Lie subalgebra $\mathfrak{g}^\omega \coloneqq \{x \in \mathfrak{g} \mid \omega(x) = x\}$ is also a simple Lie algebra. Fix a complete set $\breve{I}$ of representatives for the $\omega$-orbits in $I$; the set $\breve{I}$ is identified with an index set for the vertices of the Dynkin diagram of $\mathfrak{g}^\omega$. Then, there exists a natural injective group homomorphism $\Theta \colon \breve{W} \hookrightarrow W$ from the Weyl group of $\mathfrak{g}^\omega$ to that of $\mathfrak{g}$. If ${\bf i} = (i_1, \ldots, i_r) \in \breve{I}^r$ is a reduced word for $w \in \breve{W}$, then \[\Theta({\bf i}) \coloneqq (i_{1, 1}, \ldots, i_{1, m_{i_1}}, \ldots, i_{r, 1}, \ldots, i_{r, m_{i_r}}) \in I^{m_{i_1} + \cdots + m_{i_r}}\] is a reduced word for $\Theta(w)$, where we set $m_i \coloneqq \min \{k \in \z_{>0} \mid \omega^k(i) = i\}$ for $i \in \breve{I}$ and $i_{k, l} \coloneqq \omega^{l-1}(i_k)$ for $1 \le k \le r$, $1 \le l \le m_{i_k}$. Let $\omega^\ast \colon \mathfrak{t}^\ast \xrightarrow{\sim} \mathfrak{t}^\ast$ be the dual of the $\mathbb{C}$-linear automorphism $\omega \colon \mathfrak{t} \xrightarrow{\sim} \mathfrak{t}$, and set $(\mathfrak{t}^\ast)^0 \coloneqq \{\lambda \in \mathfrak{t}^\ast \mid \omega^\ast(\lambda) = \lambda\}$. Note that an element $\lambda \in P_+ \cap (\mathfrak{t}^\ast)^0$ naturally induces a weight $\hat{\lambda}$ for $\mathfrak{g}^\omega$. Now, for $w \in \breve{W}$ and $\lambda \in P_+ \cap (\mathfrak{t}^\ast)^0$, let $X(w)$ (resp., $X(\Theta(w))$) be the corresponding Schubert variety, and $\mathcal{L}_{\hat{\lambda}}$ (resp., $\mathcal{L}_\lambda$) the corresponding line bundle on $X(w)$ (resp., $X(\Theta(w))$). Also, let $\Delta_{\bf i} ^{(\hat{\lambda}, w)}, \Delta_{\Theta({\bf i})} ^{(\lambda, \Theta(w))}$ (resp., $\widetilde{\Delta}_{\bf i} ^{(\hat{\lambda}, w)}, \widetilde{\Delta}_{\Theta({\bf i})} ^{(\lambda, \Theta(w))}$) denote Littelmann's string polytopes (resp., Nakashima-Zelevinsky's polyhedral realizations) corresponding to $w \in \breve{W}$ and $\lambda \in P_+ \cap (\mathfrak{t}^\ast)^0$; see Definition \ref{d:representation-theoretic polytopes} for the definitions. Kaveh \cite{Kav} (resp., the author and Naito \cite{FN}) proved that 
\begin{align*}
&\Delta_{\bf i} ^{(\hat{\lambda}, w)} = -\Delta(X(w), \mathcal{L}_{\hat{\lambda}}, v_{\bf i}, \tau_{\hat{\lambda}}),\ \Delta_{\Theta({\bf i})} ^{(\lambda, \Theta(w))} = -\Delta(X(\Theta(w)), \mathcal{L}_\lambda, v_{\Theta({\bf i})}, \tau_\lambda)\\
({\rm resp.},\ &\widetilde{\Delta}_{\bf i} ^{(\hat{\lambda}, w)} = -\Delta(X(w), \mathcal{L}_{\hat{\lambda}}, \tilde{v}_{\bf i}, \tau_{\hat{\lambda}}),\ \widetilde{\Delta}_{\Theta({\bf i})} ^{(\lambda, \Theta(w))} = -\Delta(X(\Theta(w)), \mathcal{L}_\lambda, \tilde{v}_{\Theta({\bf i})}, \tau_\lambda)) 
\end{align*}
for specific valuations $v_{\bf i}, v_{\Theta({\bf i})}$ (resp., $\tilde{v}_{\bf i}, \tilde{v}_{\Theta({\bf i})}$) and specific sections $\tau_{\hat{\lambda}}, \tau_\lambda$, where the sets on the right-hand side of these equations denote the corresponding Newton-Okounkov polytopes (see Definitions \ref{d:valuations} and \ref{d:Newton-Okounkov polytopes} for the definitions). The following is the main result of this paper.

\vspace{2mm}\begin{theor}
Define an $\mathbb{R}$-linear surjective map $\Omega_{\bf i} = \Omega_{\bf i} ^{(\omega)} \colon \mathbb{R}^{m_{i_1} + \cdots + m_{i_r}} \twoheadrightarrow \mathbb{R}^r$ by \[\Omega_{\bf i}(a_{1, 1}, \ldots, a_{1, m_{i_1}}, \ldots, a_{r, 1}, \ldots, a_{r, m_{i_r}}) \coloneqq (a_{1, 1} + \cdots + a_{1, m_{i_1}}, \ldots, a_{r, 1} + \cdots + a_{r, m_{i_r}}).\] Then the following equalities hold$:$
\begin{align*}
&\Omega_{\bf i}(\Delta(X(\Theta(w)), \mathcal{L}_\lambda, v_{\Theta({\bf i})}, \tau_\lambda)) = \Delta(X(w), \mathcal{L}_{\hat{\lambda}}, v_{\bf i}, \tau_{\hat{\lambda}}),\ {\it and}\\
&\Omega_{\bf i}(\Delta(X(\Theta(w)), \mathcal{L}_\lambda, \tilde{v}_{\Theta({\bf i})}, \tau_\lambda)) = \Delta(X(w), \mathcal{L}_{\hat{\lambda}}, \tilde{v}_{\bf i}, \tau_{\hat{\lambda}}).
\end{align*} 
\end{theor}\vspace{2mm}

In our proof of the theorem above, we use another simply-laced simple Lie algebra $\mathfrak{g}^\prime$ having a Dynkin diagram automorphism $\omega^\prime \colon I^\prime \rightarrow I^\prime$ satisfying the following conditions:
\begin{enumerate}
\item[${\rm (C)}_1$] the fixed point Lie subalgebra $(\mathfrak{g}')^{\omega'}$ is isomorphic to the orbit Lie algebra $\breve{\mathfrak{g}}$ associated to $\omega$; this condition implies that the index set $\breve{I}$ for $\breve{\mathfrak{g}}$ is identified with an index set $\breve{I}^\prime\ (= \breve{(I^\prime)})$ for $(\mathfrak{g}')^{\omega'}$;
\item[${\rm (C)}_2$] if we set $m_i ^\prime \coloneqq \min\{k \in \mathbb{Z}_{>0} \mid (\omega^\prime)^k(i) = i\}$, $i \in \breve{I}^\prime$, then the product $L \coloneqq m_i \cdot m_i ^\prime$ is independent of the choice of $i \in \breve{I} \simeq \breve{I}^\prime$.
\end{enumerate}
Let ${\bf i} = (i_1, \ldots, i_r) \in \breve{I}^r \simeq (\breve{I}^\prime)^r$ be a reduced word. It is known that $P_+ \cap (\mathfrak{t}^\ast)^0$ is identified with the set of dominant weights for the orbit Lie algebra $\breve{\mathfrak{g}}$ associated to $\omega$; let $\breve{\lambda}$ denote the dominant weight for $\breve{\mathfrak{g}}$ corresponding to $\lambda \in P_+ \cap (\mathfrak{t}^\ast)^0$. Now we define an $\mathbb{R}$-linear injective map $\Upsilon_{\bf i} = \Upsilon_{\bf i} ^{(\omega)} \colon \mathbb{R}^r \hookrightarrow \mathbb{R}^{m_{i_1} + \cdots + m_{i_r}}$ by 
\[\Upsilon_{\bf i}(a_1, \ldots, a_r) \coloneqq (\underbrace{a_1, \ldots, a_1}_{m_{i_1}}, \ldots, \underbrace{a_r, \ldots, a_r}_{m_{i_r}}).\] 
By using the theory of crystal bases, we see that Littelmann's string polytope (resp., Nakashima-Zelevinsky's polyhedral realization) for $\breve{\mathfrak{g}}$ with respect to $\breve{\lambda}$ and ${\bf i}$ is identified with a slice of $\Delta_{\Theta({\bf i})} ^{(\lambda, \Theta(w))}$ (resp., $\widetilde{\Delta}_{\Theta({\bf i})} ^{(\lambda, \Theta(w))}$) through $\Upsilon_{\bf i}$ (see Corollary \ref{c:folding of polytopes(crystal)} for more details). Hence we obtain the following diagram:
\begin{align*}
\xymatrix{                   
      & \r^{m_{i_1} + \cdots + m_{i_r}} \ar@{->>}[rd]^-{\Omega_{\bf i} ^{(\omega)}} &   \\
\r^r \ar@{^{(}->}[ru]^-{\Upsilon_{\bf i} ^{(\omega)}} &                            & \ar@{^{(}->}[ld]^-{\Upsilon_{\bf i} ^{(\omega^\prime)}} \r^r,\\
                         & \ar@{->>}[lu]^-{\Omega_{\bf i} ^{(\omega^\prime)}} \r^{m_{i_1} ^\prime + \cdots +m_{i_r} ^\prime}&   
} 
\end{align*}
in which the composite maps $\Omega_{\bf i} ^{(\omega)} \circ \Upsilon_{\bf i} ^{(\omega)} \circ \Omega_{\bf i} ^{(\omega^\prime)} \circ \Upsilon_{\bf i} ^{(\omega^\prime)}$ and $\Omega_{\bf i} ^{(\omega^\prime)} \circ \Upsilon_{\bf i} ^{(\omega^\prime)} \circ \Omega_{\bf i} ^{(\omega)} \circ \Upsilon_{\bf i} ^{(\omega)}$ are both identical to $L \cdot {\rm id}_{\r^r}$, where $L$ is the positive integer in ${\rm (C)}_2$. This diagram plays an important role in our proof of the Theorem above. If $\mathfrak{g}$ is of type $A_{2n-1}$ and $\omega$ is its Dynkin diagram automorphism of order two, then $\mathfrak{g}^\omega$ is of type $C_n$ and $(\mathfrak{g}^\prime, \omega^\prime)$ is given uniquely by the pair of the simple Lie algebra of type $D_{n+1}$ and its Dynkin diagram automorphism of order two; the fixed point Lie subalgebra $(\mathfrak{g}^\prime)^{\omega^\prime}$ is of type $B_n$. Thus the diagram above relates Newton-Okounkov polytopes of Schubert varieties of types $A$, $B$, $C$, and $D$. A remarkable fact is that the composite map $\Omega_{\bf i} \circ \Upsilon_{\bf i}$ is identical to the map coming from a similarity of crystal bases. This gives a new interpretation of the similarity of crystal bases in terms of the folding procedure.

This paper is organized as follows. In Section 2, we recall some basic facts about Littelmann's string polytopes and Nakashima-Zelevinsky's polyhedral realizations. In Section 3, we review main results of \cite{FN} and \cite{Kav}. Section 4 is devoted to the study of the folding procedure for crystal bases. In Section 5, we prove the Theorem above. In Section 6, we study the relation with a similarity of crystal bases. Finally, we mention that our arguments in this paper are naturally extended to symmetrizable Kac-Moody algebras; in Appendix A, we give the list of nontrivial pairs of automorphisms of simply-laced affine Dynkin diagrams satisfying conditions ${\rm (C)}_1$ and ${\rm (C)}_2$ above.

\vspace{2mm}\begin{ack}\normalfont
The author is greatly indebted to his supervisor Satoshi Naito for fruitful discussions and numerous helpful suggestions. The author would also like to thank Hironori Oya for suggesting the relation with a similarity of crystal bases.
\end{ack}\vspace{2mm}

\section{Littelmann's string polytopes and Nakashima-Zelevinsky's polyhedral realizations}

In this section, we consider Littelmann's string polytopes and Nakashima-Zelevinsky's polyhedral realizations, which are the main objects of our study. We first recall some basic facts about crystal bases, following \cite{Kas1, Kas2, Kas3, Kas4}. Let $G$ be a connected, simply-connected simple algebraic group over $\c$, $\mathfrak{g}$ its Lie algebra, $W$ the Weyl group, $T \subset G$ a maximal torus, and $I$ an index set for the vertices of the Dynkin diagram of $\mathfrak{g}$. Let $\mathfrak{t} \subset \mathfrak{g}$ denote the Lie algebra of $T$, $\mathfrak{t}^\ast \coloneqq {\rm Hom}_\mathbb{C} (\mathfrak{t}, \mathbb{C})$ the dual space of $\mathfrak{t}$, and $\langle \cdot, \cdot \rangle \colon \mathfrak{t}^\ast \times \mathfrak{t} \rightarrow \mathbb{C}$ the canonical pairing. Denote by $P \subset \mathfrak{t}^\ast$ the weight lattice for $\mathfrak{g}$, by $P_+ \subset P$ the set of dominant integral weights, by $\{\alpha_i \mid i \in I\} \subset \mathfrak{t}^\ast$ the set of simple roots, and by $\{h_i \mid i \in I\} \subset \mathfrak{t}$ the set of simple coroots. For an indeterminate $q$, we define $q_i \in \q(q)$, $i \in I$, by:
\begin{align*}
&q_i = \begin{cases}
q^3\quad{\rm if}\ \mathfrak{g}\ {\rm is\ of\ type}\ G_2\ {\rm and}\ \alpha_i\ {\rm is\ a\ long\ root},\\
q^2\quad{\rm if}\ \mathfrak{g}\ {\rm is\ of\ type}\ B_n, C_n,\ n \ge 2,\ {\rm or}\ F_4,\ {\rm and}\ \alpha_i\ {\rm is\ a\ long\ root},\\
q\quad\ {\rm otherwise}.
\end{cases}
\end{align*}
Let $U_q(\mathfrak{g})$ be the quantized enveloping algebra of $\mathfrak{g}$ over $\mathbb{Q}(q)$ with generators $\{e_i, f_i, t_i, t_i ^{-1} \mid i \in I\}$, and $U_q (\mathfrak{u}^-)$ the $\mathbb{Q}(q)$-subalgebra of $U_q(\mathfrak{g})$ generated by $\{f_i \mid i \in I\}$. Denote by $\mathcal{B}(\infty)$ the crystal basis of $U_q (\mathfrak{u}^-)$ with $b_\infty \in \mathcal{B}(\infty)$ the element corresponding to $1 \in U_q (\mathfrak{u}^-)$, and by $\tilde{e}_i, \tilde{f}_i \colon \mathcal{B}(\infty) \cup \{0\} \rightarrow \mathcal{B}(\infty) \cup \{0\}$ for $i \in I$ the Kashiwara operators. 

\vspace{2mm}\begin{defi}\normalfont
Define a $\mathbb{Q}(q)$-algebra anti-involution $\ast$ on $U_q (\mathfrak{g})$ by: \[e_i ^\ast = e_i,\ f_i ^\ast = f_i,\ t_i ^\ast = t_i ^{-1}\] for $i \in I$; we see by \cite[Theorem 2.1.1]{Kas4} that this induces an involution $\ast \colon \mathcal{B}(\infty) \rightarrow \mathcal{B}(\infty)$, called {\it Kashiwara's involution}. 
\end{defi}\vspace{2mm}

For $\lambda \in P_+$, denote by $V_q(\lambda)$ the irreducible highest weight $U_q(\mathfrak{g})$-module with highest weight $\lambda$ over $\mathbb{Q}(q)$, and by $v_{q, \lambda} \in V_q(\lambda)$ the highest weight vector. Let $\mathcal{B}(\lambda)$ denote the crystal basis of $V_q(\lambda)$ with $b_\lambda \in \mathcal{B}(\lambda)$ the element corresponding to $v_{q, \lambda} \in V_q(\lambda)$, and $\tilde{e}_i, \tilde{f}_i \colon \mathcal{B}(\lambda) \cup \{0\} \rightarrow \mathcal{B}(\lambda) \cup \{0\}$ for $i \in I$ the Kashiwara operators. Define maps $\varepsilon_i, \varphi_i \colon \mathcal{B}(\infty) \rightarrow \mathbb{Z}$ and $\varepsilon_i, \varphi_i \colon \mathcal{B}(\lambda) \rightarrow \mathbb{Z}$ for $i \in I$ by 
\begin{align*}
&\varepsilon_i(b) \coloneqq \max\{k \in \mathbb{Z}_{\ge 0} \mid \tilde{e}_i ^k b \neq 0\},\ \varphi_i (b) \coloneqq \varepsilon_i (b) + \langle{\rm wt}(b), h_i\rangle\ {\rm for}\ b \in \mathcal{B}(\infty),\ {\rm and}\\
&\varepsilon_i(b) \coloneqq \max\{k \in \mathbb{Z}_{\ge 0} \mid \tilde{e}_i ^k b \neq 0\},\ \varphi_i (b) \coloneqq \max\{k \in \mathbb{Z}_{\ge 0} \mid \tilde{f}_i ^k b \neq 0\}\ {\rm for}\ b \in \mathcal{B}(\lambda).
\end{align*}

\vspace{2mm}\begin{prop}[{\cite[Theorem 5]{Kas2}}]
For $\lambda \in P_+$, let $\pi_\lambda \colon U_q(\mathfrak{u}^-) \twoheadrightarrow V_q(\lambda)$ denote the surjective $U_q(\mathfrak{u}^-)$-module homomorphism given by $u \mapsto u v_{q, \lambda}$.
\begin{enumerate}
\item[{\rm (1)}] The homomorphism $\pi_\lambda$ induces a surjective map $\mathcal{B}(\infty) \twoheadrightarrow \mathcal{B}(\lambda) \cup \{0\}$ $($denoted also by $\pi_\lambda)$. For \[\widetilde{\mathcal{B}}(\lambda) \coloneqq \{b \in \mathcal{B}(\infty) \mid \pi_\lambda(b) \neq 0\},\] the restriction map $\pi_\lambda \colon \widetilde{\mathcal{B}}(\lambda) \rightarrow \mathcal{B}(\lambda)$ is bijective.
\item[{\rm (2)}] $\tilde{f}_i \pi_\lambda(b) = \pi_\lambda (\tilde{f}_i b)$ for all $i \in I$ and $b \in \mathcal{B}(\infty)$.
\item[{\rm (3)}] $\tilde{e}_i \pi_\lambda(b) = \pi_\lambda (\tilde{e}_i b)$ for all $i \in I$ and $b \in \widetilde{\mathcal{B}}(\lambda)$. 
\item[{\rm (4)}] $\varepsilon_i (\pi_\lambda(b)) = \varepsilon_i (b)$ and $\varphi_i (\pi_\lambda (b)) = \varphi_i (b) + \langle\lambda, h_i\rangle$ for all $i \in I$ and $b \in \widetilde{\mathcal{B}}(\lambda)$. 
\end{enumerate}
\end{prop}\vspace{2mm}

\begin{defi}\normalfont
Let ${\bf i} = (i_1, \ldots, i_r) \in I^r$ be a reduced word for $w \in W$, and $\lambda \in P_+$. By \cite[Propositions 3.2.3 and 3.2.5]{Kas4}, the subsets 
\begin{align*}
&\mathcal{B}_w(\infty) \coloneqq \{\tilde{f}_{i_1} ^{a_1} \cdots \tilde{f}_{i_r} ^{a_r} b_\infty \mid a_1, \ldots, a_r \in \mathbb{Z}_{\ge 0}\} \subset \mathcal{B}(\infty)\ {\rm and}\\
&\mathcal{B}_w(\lambda) \coloneqq \{\tilde{f}_{i_1} ^{a_1} \cdots \tilde{f}_{i_r} ^{a_r} b_\lambda \mid a_1, \ldots, a_r \in \mathbb{Z}_{\ge 0}\} \setminus \{0\} \subset \mathcal{B}(\lambda)
\end{align*}
are independent of the choice of a reduced word ${\bf i}$. These subsets $\mathcal{B}_w (\infty), \mathcal{B}_w(\lambda)$ are called {\it Demazure crystals}. 
\end{defi}\vspace{2mm}

\begin{prop}[{see \cite[Proposition 3.2.5]{Kas4}}]\label{Demazure crystal}
For $\lambda \in P_+$ and $w \in W$, the equality $\pi_\lambda(\mathcal{B}_w(\infty)) = \mathcal{B}_w(\lambda) \cup \{0\}$ holds$;$ hence $\pi_\lambda$ induces a bijective map $\pi_\lambda \colon \widetilde{\mathcal{B}}_w(\lambda) \rightarrow \mathcal{B}_w(\lambda)$, where $\widetilde{\mathcal{B}}_w(\lambda) \coloneqq \mathcal{B}_w(\infty) \cap \widetilde{\mathcal{B}}(\lambda)$.
\end{prop}\vspace{2mm}

In the theory of crystal bases, it is important to give their concrete parametrizations. In this paper, we use two parametrizations: Littelmann's string parametrization and the Kashiwara embedding.

\vspace{2mm}\begin{defi}\normalfont
Let ${\bf i} = (i_1, \ldots, i_r) \in I^r$ be a reduced word for $w \in W$, and $b \in \mathcal{B}_w(\infty)$. Define $\Phi_{\bf i} (b) = (a_1, \ldots, a_r) \in \mathbb{Z}_{\ge 0} ^r$ by
\begin{align*}
&a_1 \coloneqq \max\{a \in \mathbb{Z}_{\ge 0} \mid \tilde{e}_{i_1} ^a b \neq 0\},\\
&a_2 \coloneqq \max\{a \in \mathbb{Z}_{\ge 0} \mid \tilde{e}_{i_2} ^a \tilde{e}_{i_1} ^{a_1} b \neq 0\},\\
&\ \vdots\\
&a_r \coloneqq \max\{a \in \mathbb{Z}_{\ge 0} \mid \tilde{e}_{i_r} ^a \tilde{e}_{i_{r-1}} ^{a_{r-1}} \cdots \tilde{e}_{i_1} ^{a_1} b \neq 0\}.
\end{align*}
The $\Phi_{\bf i}(b)$ is called {\it Littelmann's string parametrization} of $b$ with respect to ${\bf i}$ (see \cite[Sect.\ 1]{Lit}). 
\end{defi}\vspace{2mm}

By \cite[Proposition 3.3.1]{Kas4}, we have $\mathcal{B}_w(\infty)^\ast = \mathcal{B}_{w^{-1}}(\infty)$; hence the map $\Phi_{{\bf i}^{\rm op}} \circ \ast \colon \mathcal{B}_w(\infty) \rightarrow \mathbb{Z}_{\ge 0} ^r$ is well-defined, where ${\bf i}^{\rm op} \coloneqq (i_r, \ldots, i_1)$ is a reduced word for $w^{-1}$. 

\vspace{2mm}\begin{defi}\normalfont
Let ${\bf i} = (i_1, \ldots, i_r) \in I^r$ be a reduced word for $w \in W$. Define a map $\Psi_{\bf i} \colon \mathcal{B}_w(\infty) \rightarrow \mathbb{Z}_{\ge 0} ^r$ by $\Psi_{\bf i}(b) \coloneqq \Phi_{{\bf i}^{\rm op}}(b^\ast)^{\rm op}$ for $b \in \mathcal{B}_w(\infty)$, where ${\bf a}^{\rm op} \coloneqq (a_r, \ldots, a_1)$ for ${\bf a} = (a_1, \ldots, a_r) \in \mathbb{Z}_{\ge 0} ^r$. The map $\Psi_{\bf i}$ is called the {\it Kashiwara embedding} of $\mathcal{B}_w(\infty)$ (see \cite[Sects.\ 2 and 3]{Kas4}).
\end{defi}

\vspace{2mm}\begin{rem}\normalfont\label{parametrization for modules}
By the bijective map $\pi_\lambda \colon \widetilde{\mathcal{B}}_w(\lambda) \xrightarrow{\sim} \mathcal{B}_w(\lambda)$ in Proposition \ref{Demazure crystal}, the maps $\Phi_{\bf i}$ and $\Psi_{\bf i}$ can be thought of as ones from $\mathcal{B}_w(\lambda)$, called {\it Littelmann's string parametrization} of $\mathcal{B}_w(\lambda)$ and the {\it Kashiwara embedding} of $\mathcal{B}_w(\lambda)$, respectively.
\end{rem}

\vspace{2mm}\begin{defi}\normalfont\label{d:representation-theoretic polytopes}
Let ${\bf i} = (i_1, \ldots, i_r) \in I^r$ be a reduced word for $w \in W$, and $\lambda \in P_+$. Define a subset $\mathcal{S}_{\bf i} ^{(\lambda, w)} \subset \mathbb{Z}_{>0} \times \mathbb{Z}^r$ by \[\mathcal{S}_{\bf i} ^{(\lambda, w)} \coloneqq \bigcup_{k>0} \{(k, \Phi_{\bf i}(b)) \mid b \in \widetilde{\mathcal{B}}_w (k\lambda)\},\] and denote by $\mathcal{C}_{\bf i} ^{(\lambda, w)} \subset \mathbb{R}_{\ge 0} \times \mathbb{R}^r$ the smallest real closed cone containing $\mathcal{S}_{\bf i} ^{(\lambda, w)}$. Then, we define a subset $\Delta_{\bf i} ^{(\lambda, w)} \subset \mathbb{R}^r$ by \[\Delta_{\bf i} ^{(\lambda, w)} \coloneqq \{{\bf a} \in \mathbb{R}^r \mid (1, {\bf a}) \in \mathcal{C}_{\bf i} ^{(\lambda, w)}\}.\] This subset $\Delta_{\bf i} ^{(\lambda, w)}$ is called {\it Littelmann's string polytope} for $\mathcal{B}_w(\lambda)$ with respect to ${\bf i}$ (see \cite[Definition 3.5]{Kav} and \cite[Sect.\ 1]{Lit}). Also, by replacing $\Phi_{\bf i}$ with $\Psi_{\bf i}$ in the definitions of $\mathcal{S}_{\bf i} ^{(\lambda, w)}$, $\mathcal{C}_{\bf i} ^{(\lambda, w)}$, and $\Delta_{\bf i} ^{(\lambda, w)}$, we obtain $\widetilde{\mathcal{S}}_{\bf i} ^{(\lambda, w)} \subset \mathbb{Z}_{>0} \times \mathbb{Z}^r$, $\widetilde{\mathcal{C}}_{\bf i} ^{(\lambda, w)} \subset \mathbb{R}_{\ge 0} \times \mathbb{R}^r$, and $\widetilde{\Delta}_{\bf i} ^{(\lambda, w)} \subset \mathbb{R}^r$. We call the subset $\widetilde{\Delta}_{\bf i} ^{(\lambda, w)}$ {\it Nakashima-Zelevinsky's polytope} for $\mathcal{B}_w(\lambda)$ with respect to ${\bf i}$ (see \cite[Sect.\ 2.3]{FN}, \cite[Sects.\ 3 and 4]{Nak1}, \cite[Sect.\ 3.1]{Nak2}, and \cite[Sect.\ 3]{NZ}).
\end{defi}\vspace{2mm}

A subset $\mathcal{C} \subset \mathbb{R}_{\ge 0} \times \mathbb{R}^r$ is said to be a {\it rational convex polyhedral cone} if there exists a finite number of rational points ${\bf a}_1, \ldots, {\bf a}_l \in \mathbb{Q}_{\ge 0} \times \mathbb{Q}^r$ such that $\mathcal{C} = \mathbb{R}_{\ge 0}{\bf a}_1 + \cdots + \mathbb{R}_{\ge 0} {\bf a}_l$. A subset $\Delta \subset \mathbb{R}^r$ is said to be a {\it rational convex polytope} if it is the convex hull of a finite number of rational points.

\vspace{2mm}\begin{prop}[{see \cite[Sect.\ 3.2 and Theorem 3.10]{BZ}, \cite[Corollary 4.3]{FN} and \cite[Sect.\ 1]{Lit}}]\label{string lattice points}
Let ${\bf i} = (i_1, \ldots, i_r) \in I^r$ be a reduced word for $w \in W$, and $\lambda \in P_+$. 
\begin{enumerate}
\item[{\rm (1)}] The real closed cones $\mathcal{C}_{\bf i} ^{(\lambda, w)}$ and $\widetilde{\mathcal{C}}_{\bf i} ^{(\lambda, w)}$ are both rational convex polyhedral cones$;$ in addition, the following equalities hold$:$ \[\mathcal{S}_{\bf i} ^{(\lambda, w)} = \mathcal{C}_{\bf i} ^{(\lambda, w)} \cap (\mathbb{Z}_{>0} \times \mathbb{Z}^r),\ \widetilde{\mathcal{S}}_{\bf i} ^{(\lambda, w)} = \widetilde{\mathcal{C}}_{\bf i} ^{(\lambda, w)} \cap (\mathbb{Z}_{>0} \times \mathbb{Z}^r).\]
\item[{\rm (2)}] The sets $\Delta_{\bf i} ^{(\lambda, w)}$ and $\widetilde{\Delta}_{\bf i} ^{(\lambda, w)}$ are both rational convex polytopes$;$ in addition, the following equalities hold$:$ \[\Phi_{\bf i} (\widetilde{\mathcal{B}}_w(\lambda)) = \Delta_{\bf i} ^{(\lambda, w)} \cap \mathbb{Z}^r,\ \Psi_{\bf i} (\widetilde{\mathcal{B}}_w(\lambda)) = \widetilde{\Delta}_{\bf i} ^{(\lambda, w)} \cap \mathbb{Z}^r.\]
\end{enumerate}
\end{prop}\vspace{2mm}

\begin{rem}\normalfont
By \cite[Theorem 3.10]{BZ} and \cite[Sect.\ 1]{Lit}, we obtain a system of explicit linear inequalities defining Littelmann's string polytope $\Delta_{\bf i} ^{(\lambda, w)}$. In addition, under a certain positivity assumption on ${\bf i}$, Nakashima \cite{Nak1, Nak2} gave a system of explicit linear inequalities defining Nakashima-Zelevinsky's polytope $\widetilde{\Delta}_{\bf i} ^{(\lambda, w)}$ (see also \cite[Corollary 5.3]{FN}).
\end{rem}\vspace{2mm}

\begin{rem}\normalfont
In \cite{FN, FO}, the polytope $\widetilde{\Delta}_{\bf i} ^{(\lambda, w)}$ is called Nakashima-Zelevinsky's polyhedral realization. However, the word ``polyhedral realization'' is originally used in \cite{Nak1, Nak2, NZ} to mean the realization of a crystal basis as the lattice points in an explicit rational convex polyhedral cone or an explicit rational convex polytope. Hence the terminology in \cite{FN, FO} is slightly inaccurate.
\end{rem}\vspace{2mm}

\section{Perfect bases and Newton-Okounkov polytopes}

In this section, we recall the definition of Newton-Okounkov polytopes of Schubert varieties, following \cite{HK, Kav, KK1, KK2}. 

Let us fix a Borel subgroup $B \subset G$, and denote by $B^- \subset G$ the opposite Borel subgroup. Then, the full flag variety is defined to be the quotient space $G/B$. For $w \in W$, let $X(w) \subset G/B$ denote the Schubert variety corresponding to $w$, that is, $X(w)$ is the Zariski closure of $B\widetilde{w}B/B$ in $G/B$, where $\widetilde{w} \in G$ denotes a lift for $w$; note that $X(w)$ is independent of the choice of $\widetilde{w}$. It is well-known that $X(w)$ is a normal projective variety of complex dimension $\ell(w)$; here, $\ell(w)$ denotes the length of $w$. Also, for a given $\lambda \in P_+$, we define a line bundle $\mathcal{L}_\lambda$ on $G/B$ by \[\mathcal{L}_\lambda \coloneqq (G \times \mathbb{C})/B,\] where $B$ acts on $G \times \mathbb{C}$ on the right as follows: \[(g, c) \cdot b = (g b, \lambda(b) c)\] for $g \in G$, $c \in \mathbb{C}$, and $b \in B$. By restricting this bundle, we obtain a line bundle on $X(w)$, which we denote by the same symbol $\mathcal{L}_\lambda$. Let $U^-$ denote the unipotent radical of $B^-$ with Lie algebra $\mathfrak{u}^-$, and regard $U^-$ as an affine open subvariety of $G/B$ by the following open immersion: \[U^- \hookrightarrow G/B,\ u \mapsto u \bmod B.\] Then we consider the set-theoretic intersection $U^- \cap X(w)$ in $G/B$. Since this intersection is an open subset of $X(w)$, it inherits an open subvariety structure from $X(w)$; note that it coincides with the variety structure on $U^- \cap X(w)$ as a closed subvariety of $U^-$ (see \cite[Sect.\ 2]{FO}). 

Let $\mathfrak{b} \subset \mathfrak{g}$ be the Lie algebra of $B$, and $E_i, F_i, h_i \in \mathfrak{g}$, $i \in I$, the Chevalley generators such that $\{E_i, h_i \mid i \in I\} \subset \mathfrak{b}$ and $\{F_i \mid i \in I\} \subset \mathfrak{u}^-$. We set $[k]_i ! \coloneqq [k]_i [k-1]_i \cdots [1]_i$ for $i \in I$, $k \in \mathbb{Z}_{> 0}$, and $[0]_i ! \coloneqq 1$, where
\begin{align*}
[k]_i \coloneqq \frac{q_i ^k - q_i ^{-k}}{q_i - q_i ^{-1}}\ {\rm for}\ i \in I,\ k \in \mathbb{Z}_{>0}.
\end{align*}
Also, let $U_{q, \z}(\mathfrak{u}^-)$ denote the $\z[q^{\pm 1}]$-subalgebra of $U_q(\mathfrak{u}^-)$ generated by $\{f_i ^{(k)} \mid i \in I,\ k \in \z_{\ge 0}\}$, where $f_i ^{(k)} \coloneqq f_i ^k/[k]_i !$. Then, the $\c$-algebra $\mathbb{C} \otimes_{\mathbb{Z}[q^{\pm 1}]} U_{q, \mathbb{Z}}(\mathfrak{u}^-)$ is isomorphic to the universal enveloping algebra $U(\mathfrak{u}^-)$ of $\mathfrak{u}^-$ by $1 \otimes f_i^{(k)} \mapsto F_i ^k/k!$, where the $\mathbb{Z}[q^{\pm 1}]$-module structure on $\mathbb{C}$ is given by $q \mapsto 1$; hence this process is called the {\it specialization at $q=1$}. We define a $\c$-algebra anti-involution $\ast$ on $U(\mathfrak{u}^-)$ by $F_i ^\ast \coloneqq F_i$ for all $i \in I$. The algebra $U(\mathfrak{u}^-)$ has a Hopf algebra structure given by the following coproduct $\Delta$, counit $\varepsilon$, and antipode $S$:
\[\Delta(F_i) = F_i \otimes 1 + 1 \otimes F_i,\ \varepsilon(F_i) = 0,\ {\rm and}\ S(F_i) = -F_i\]
for $i \in I$. In addition, we regard $U(\mathfrak{u}^-)$ as a multigraded $\mathbb{C}$-algebra: \[U(\mathfrak{u}^-) = \bigoplus_{{\bf d} \in \mathbb{Z}^I _{\ge 0}} U(\mathfrak{u}^-)_{\bf d},\] where the homogeneous component $U(\mathfrak{u}^-)_{\bf d}$ for ${\bf d} = (d_i)_{i \in I} \in \mathbb{Z}^I _{\ge 0}$ is defined to be the $\mathbb{C}$-subspace of $U(\mathfrak{u}^-)$ spanned by all those elements $F_{j_1} \cdots F_{j_{|{\bf d}|}}$ such that the cardinality of $\{1 \le k \le |{\bf d}| \mid j_k = i\}$ is equal to $d_i$ for every $i \in I$; here we set $|{\bf d}| \coloneqq \sum_{i \in I} d_i$. Let \[U(\mathfrak{u}^-)^\ast _{\rm gr} \coloneqq \bigoplus_{{\bf d} \in \mathbb{Z}^I _{\ge 0}} {\rm Hom}_\c(U(\mathfrak{u}^-)_{\bf d}, \c)\] be the graded dual of $U(\mathfrak{u}^-)$ endowed with the dual Hopf algebra structure. Note that the coordinate ring $\mathbb{C}[U^-]$ has a Hopf algebra structure given by the following coproduct $\Delta$, counit $\varepsilon$, and antipode $S$: \[\Delta(f) (u_1, u_2) = f(u_1 u_2),\ \varepsilon(f) = f(e)\ {\rm and}\ S(f) (u) = f(u^{-1})\] for $f \in \mathbb{C}[U^-]$ and $u, u_1, u_2 \in U^-$, where $e \in U^-$ denotes the identity element. It is known that this Hopf algebra $\c[U^-]$ is isomorphic to the dual Hopf algebra $U(\mathfrak{u}^-)_{\rm gr} ^\ast$ (see, for instance, \cite[Proposition 5.1]{GLS}).

\vspace{2mm}\begin{defi}[{see \cite[Definition 5.30]{BK}, \cite[Definition 2.5]{KOP1}, and \cite[Sect.\ 4.2]{KOP2}}]\label{definition of perfect bases}\normalfont
A $\c$-basis ${\bf B}^{\rm low} \subset U(\mathfrak{u}^-)$ is said to be ({\it lower}) {\it perfect} if there exists a bijection $\Xi^{\rm low} \colon \mathcal{B}(\infty) \xrightarrow{\sim} {\bf B}^{\rm low}$ satisfying the following conditions:
\begin{enumerate}
\item[{\rm (i)}] ${\bf B}^{\rm low} = \bigcup_{{\bf d} \in \z^I _{\ge 0}} {\bf B}^{\rm low} _{\bf d}$, where ${\bf B}^{\rm low} _{\bf d} \coloneqq {\bf B}^{\rm low} \cap U(\mathfrak{u}^-)_{\bf d}$ for ${\bf d} \in \z^I _{\ge 0}$,
\item[{\rm (ii)}] $\Xi^{\rm low}(b_\infty) = 1$,
\item[{\rm (iii)}] for all $i \in I$, $b \in \mathcal{B}(\infty)$ and $k \in \mathbb{Z}_{\ge 0}$, \[F_i ^{(k)} \cdot \Xi^{\rm low}(b) \in \c^\times \Xi^{\rm low}(\tilde{f}_i ^k b) + \sum_{\substack{b^\prime \in \mathcal{B}(\infty);\ {\rm wt}(b^\prime) = {\rm wt}(\tilde{f}_i ^k b),\\ \varepsilon_i (b^\prime) > \varepsilon_i (\tilde{f}_i ^k b)}} \c \Xi^{\rm low}(b^\prime),\] where $\c^\times \coloneqq \c \setminus \{0\}$.
\end{enumerate} 
In addition, we always impose the following $\ast$-stability condition on a perfect basis:
\begin{enumerate}
\item[{\rm (iv)}] $({\bf B}^{\rm low})^{\ast}={\bf B}^{\rm low}$.
\end{enumerate}
\end{defi}\vspace{2mm}

\begin{prop}[{\cite[Proposition 3.10]{FO}}]\label{p:*_property}
The equality $\Xi^{\rm low}(b)^\ast =\Xi^{\rm low}(b^\ast)$ holds for each $b\in \mathcal{B}(\infty)$.
\end{prop}\vspace{2mm}

\begin{ex}\label{e:perfect_global}\normalfont
Lusztig \cite{Lus_can, Lus_quivers, Lus1} and Kashiwara \cite{Kas2} constructed a specific $\mathbb{Z}[q^{\pm 1}]$-basis $\{G_q ^{\rm low}(b) \mid b \in \mathcal{B}(\infty)\}$ of $U_{q, \mathbb{Z}}(\mathfrak{u}^-)$, called the {\it canonical basis} or the {\it lower global basis}. The specialization $\{G^{\rm low}(b) \mid b \in \mathcal{B}(\infty)\} \subset U(\mathfrak{u}^-)$ of $\{G_q ^{\rm low}(b) \mid b \in \mathcal{B}(\infty)\}$ at $q = 1$ is a perfect basis by \cite[Proposition 5.3.1]{Kas3} and \cite[Theorem 2.1.1]{Kas4} (see also \cite[Proposition 2.8]{FN}). 
\end{ex}\vspace{2mm}

\begin{ex}\normalfont
When $\mathfrak{g}$ is simply-laced, Lusztig \cite{Lus2} constructed a specific $\c$-basis of $U(\mathfrak{u}^-)$, called the {\it semicanonical basis}. This is a perfect basis by \cite[Proof of Lemma 2.4 and Sect.\ 3]{Lus2}.
\end{ex}\vspace{2mm}

For $\lambda \in P_+$, denote by $V(\lambda)$ the irreducible highest weight $\mathfrak{g}$-module with highest weight $\lambda$ with $v_\lambda \in V(\lambda)$ the highest weight vector, and by $\pi_\lambda \colon U(\mathfrak{u}^-) \twoheadrightarrow V(\lambda)$ the surjective $U(\mathfrak{u}^-)$-module homomorphism given by $u \mapsto u v_\lambda$. We set $\Xi_{\lambda}^{\rm low}(\pi_\lambda (b)) \coloneqq \pi_\lambda(\Xi^{\rm low}(b))$ for $b \in \widetilde{\mathcal{B}}(\lambda)$.

\vspace{2mm}\begin{prop}[{see \cite[Proposition 3.14 (1)]{FO}}]\label{p:compatibility}
The set $\{\Xi_{\lambda}^{\rm low}(b) \mid b \in \mathcal{B}(\lambda)\}$ provides a $\c$-basis of $V(\lambda)$, and the element $\pi_\lambda(\Xi^{\rm low}(b))$ is identical to $0$ for $b \in \mathcal{B}(\infty) \setminus \widetilde{\mathcal{B}}(\lambda)$.
\end{prop}\vspace{2mm}

For $w \in W$, let $v_{w\lambda} \in V(\lambda)$ denote the extremal weight vector of weight $w\lambda$. The {\it Demazure module} $V_w(\lambda)$ corresponding to $w \in W$ is the $B$-submodule of $V(\lambda)$ given by \[V_w(\lambda) \coloneqq \sum_{b \in B} \mathbb{C} b v_{w\lambda}.\] By the Borel-Weil type theorem (see \cite[Corollary 8.1.26]{Kum}), we know that the space $H^0(X(w), \mathcal{L}_\lambda)$ of global sections is a $B$-module isomorphic to the dual module $V_w (\lambda)^\ast \coloneqq {\rm Hom}_\mathbb{C}(V_w(\lambda), \mathbb{C})$. We consider the following condition (D) for a perfect basis ${\bf B}^{\rm low} = \{\Xi^{\rm low}(b) \mid b \in \mathcal{B}(\infty)\}$ (see also Proposition \ref{p:compatibility}):
\begin{enumerate}
\item[(D)] the set $\{\Xi_{\lambda}^{\rm low}(b) \mid b \in \mathcal{B}_w(\lambda)\}$ is a $\c$-basis of the Demazure module $V_w(\lambda)$.
\end{enumerate}

\vspace{2mm}\begin{ex}\label{r:D_example}\normalfont
The specialization $\{G^{\rm low}(b) \mid b \in \mathcal{B}(\infty)\}$ of the lower global basis at $q=1$ and the semicanonical basis satisfy condition (D) by \cite[Proposition 3.2.3]{Kas4} and \cite[Theorem 7.1]{Sav}, respectively.
\end{ex}\vspace{2mm}

Let ${\bf B}^{\rm up} = \{\Xi^{\rm up} (b) \mid b \in \mathcal{B}(\infty)\} \subset \c[U^-] = U(\mathfrak{u}^-)_{\rm gr} ^\ast$ be the dual basis of ${\bf B}^{\rm low} = \{\Xi^{\rm low} (b) \mid b \in \mathcal{B}(\infty)\} \subset U(\mathfrak{u}^-)$. Recall that $U^- \cap X(w)$ is a Zariski closed subvariety of $U^-$. Denote by $\eta_w \colon \mathbb{C}[U^-] \twoheadrightarrow \mathbb{C}[U^- \cap X(w)]$ the restriction map, and by $\Xi_w ^{\rm up} (b) \in \mathbb{C}[U^- \cap X(w)]$ for $b \in \mathcal{B}(\infty)$ the image of $\Xi^{\rm up} (b) \in \mathbb{C}[U^-]$ under $\eta_w$. If ${\bf B}^{\rm low}$ satisfies condition (D), then let $\{\Xi^{\rm up} _{\lambda, w}(b) \mid b \in \mathcal{B}_w(\lambda)\} \subset H^0(X(w), \mathcal{L}_\lambda) = V_w(\lambda)^\ast$ denote the dual basis of $\{\Xi^{\rm low} _\lambda(b) \mid b \in \mathcal{B}_w(\lambda)\} \subset V_w(\lambda)$, and set $\tau_\lambda \coloneqq \Xi^{\rm up} _{\lambda, w}(b_\lambda)$.

\vspace{2mm}\begin{lem}[{see the proof of \cite[Lemma 4.5]{FN}}]\label{d:iota}
The section $\tau_\lambda \in H^0(X(w), \mathcal{L}_\lambda)$ does not vanish on $U^- \cap X(w)$. Hence the map $H^0(X(w), \mathcal{L}_\lambda) \rightarrow \c[U^- \cap X(w)]$, $\tau \mapsto (\tau/\tau_\lambda)|_{(U^- \cap X(w))}$, is well-defined$;$ this map is also denoted by $\iota_\lambda$. 
\end{lem}\vspace{2mm}

Since $U^- \cap X(w)$ is an open subvariety of $X(w)$, we see that the map $\iota_\lambda \colon H^0(X(w), \mathcal{L}_\lambda) \rightarrow \c[U^- \cap X(w)]$ is injective. 

\vspace{2mm}\begin{prop}[{\cite[Corollary 3.18]{FO}}]\label{vanishing}
Let ${\bf B}^{\rm up} = \{\Xi^{\rm up}(b) \mid b \in \mathcal{B}(\infty)\} \subset \c[U^-]$ be the dual basis of a perfect basis satisfying condition {\rm (D)}.
\begin{enumerate}
\item[{\rm (1)}] The following equality holds$:$ \[\c[U^- \cap X(w)] = \bigcup_{\lambda \in P_+} \iota_\lambda(H^0(X(w), \mathcal{L}_\lambda)).\] 
\item[{\rm (2)}] The element $\Xi^{\rm up} _w(b)$ is identical to $\iota_\lambda(\Xi^{\rm up} _{\lambda, w}(\pi_\lambda(b)))$ for every $b \in \widetilde{\mathcal{B}}_w(\lambda)$. 
\item[{\rm (3)}] The set $\{\Xi^{\rm up} _w(b) \mid b\in\mathcal{B}_w(\infty)\}$ provides a $\mathbb{C}$-basis of $\mathbb{C}[U^- \cap X(w)]$.
\item[{\rm (4)}] The element $\Xi^{\rm up} _w(b)$ is identical to $0$ unless $b \in \mathcal{B}_w (\infty)$.
\end{enumerate}
\end{prop}\vspace{2mm}

Let ${\bf i} = (i_1, \ldots, i_r) \in I^r$ be a reduced word for $w \in W$. It is known that the morphism $\c^r \rightarrow U^- \cap X(w)$, $(t_1, \ldots, t_r) \mapsto \exp(t_1 F_{i_1}) \cdots \exp(t_r F_{i_r}) \bmod B$, is birational. Therefore, the function field $\mathbb{C}(X(w)) = \mathbb{C}(U^- \cap X(w))$ is identified with the rational function field $\mathbb{C}(t_1, \ldots, t_r)$. 

\vspace{2mm}\begin{defi}\label{d:valuations}\normalfont
We define two lexicographic orders $<$ and $\prec$ on $\mathbb{Z}^r$ as follows: $(a_1, \ldots, a_r) < (a_1 ^\prime, \ldots, a_r ^\prime)$ (resp., $(a_1, \ldots, a_r) \prec (a_1 ^\prime, \ldots, a_r ^\prime)$) if and only if there exists $1 \le k \le r$ such that $a_1 = a_1 ^\prime, \ldots, a_{k-1} = a_{k-1} ^\prime$, $a_k < a_k ^\prime$ (resp., $a_r = a_r ^\prime, \ldots, a_{k+1} = a_{k+1} ^\prime$, $a_k < a_k ^\prime$). The lexicographic order $<$ on $\mathbb{Z}^r$ induces a total order (denoted by the same symbol $<$) on the set of all monomials in the polynomial ring $\mathbb{C}[t_1, \ldots, t_r]$ as follows: $t_1 ^{a_1} \cdots t_r ^{a_r} < t_1 ^{a_1 ^\prime} \cdots t_r ^{a_r ^\prime}$ if and only if $(a_1, \ldots, a_r) < (a_1 ^\prime, \ldots, a_r ^\prime)$. Let us define a map $v_{\bf i} \colon \mathbb{C}(X(w)) \setminus \{0\} \rightarrow \mathbb{Z}^r$ by $v_{\bf i} (f/g) \coloneqq v_{\bf i} (f) - v_{\bf i} (g)$ for $f, g \in \mathbb{C}[t_1, \ldots, t_r] \setminus \{0\}$, and by 
\begin{align*}
&v_{\bf i}(f) \coloneqq -(a_1, \ldots, a_r)\ {\rm for}\ f = c t_1 ^{a_1} \cdots t_r ^{a_r} + ({\rm lower\ terms}) \in \mathbb{C}[t_1, \ldots, t_r] \setminus \{0\},
\end{align*}
where $c \in \mathbb{C} \setminus \{0\}$, and we mean by ``lower terms'' a linear combination of monomials smaller than $t_1 ^{a_1} \cdots t_r ^{a_r}$ with respect to the total order $<$. Similarly, we define a map $\tilde{v}_{\bf i}$ by using the lexicographic order $\prec$ on $\mathbb{Z}^r$; more precisely, we set
\begin{align*}
&\tilde{v}_{\bf i}(f) \coloneqq -(a_1, \ldots, a_r)\ {\rm for}\ f = c t_1 ^{a_1} \cdots t_r ^{a_r} + ({\rm lower\ terms}) \in \mathbb{C}[t_1, \ldots, t_r] \setminus \{0\},
\end{align*}
where $c \in \mathbb{C} \setminus \{0\}$.
\end{defi}\vspace{2mm}

The map $v_{\bf i}$ is a valuation, that is, it satisfies the following conditions:
\begin{align*}
&v_{\bf i}(f \cdot g) = v_{\bf i}(f) + v_{\bf i}(g),\\
&v_{\bf i}(c \cdot f) = v_{\bf i}(f),\\
&v_{\bf i}(f + g) \ge \min\{v_{\bf i}(f), v_{\bf i}(g)\}\ {\rm with\ respect\ to\ the\ lexicographic\ order}\ < {\rm unless}\ f + g = 0
\end{align*}
for $f, g \in \mathbb{C}(X(w)) \setminus \{0\}$ and $c \in \c$. Similarly, the map $\tilde{v}_{\bf i}$ is a valuation with respect to the lexicographic order $\prec$.

\vspace{2mm}\begin{ex}\normalfont
If $r = 3$ and $f = t_1 t_2 + t_3 ^2 \in \mathbb{C}[t_1, t_2, t_3]$, then we have $v_{\bf i}(f) = -(1, 1, 0)$ and $\tilde{v}_{\bf i}(f) = -(0, 0, 2)$.
\end{ex}

\vspace{2mm}\begin{defi}\normalfont\label{d:Newton-Okounkov polytopes}
Let ${\bf i} = (i_1, \ldots, i_r) \in I^r$ be a reduced word for $w \in W$, and $\lambda \in P_+$. Take $v \in \{v_{\bf i}, \tilde{v}_{\bf i}\}$ and $\tau \in H^0(X(w), \mathcal{L}_\lambda) \setminus \{0\}$. We define a subset $S(X(w), \mathcal{L}_\lambda, v, \tau) \subset \mathbb{Z}_{>0} \times \mathbb{Z}^r$ by \[S(X(w), \mathcal{L}_\lambda, v, \tau) \coloneqq \bigcup_{k>0} \{(k, v(\sigma/\tau^k)) \mid \sigma \in H^0(X(w), \mathcal{L}_\lambda ^{\otimes k}) \setminus \{0\}\},\] and denote by $C(X(w), \mathcal{L}_\lambda, v, \tau) \subset \mathbb{R}_{\ge 0} \times \mathbb{R}^r$ the smallest real closed cone containing $S(X(w), \mathcal{L}_\lambda, v, \tau)$. Let us define a subset $\Delta(X(w), \mathcal{L}_\lambda, v, \tau) \subset \mathbb{R}^r$ by \[\Delta(X(w), \mathcal{L}_\lambda, v, \tau) \coloneqq \{{\bf a} \in \mathbb{R}^r \mid (1, {\bf a}) \in C(X(w), \mathcal{L}_\lambda, v, \tau)\};\]
this is called the {\it Newton-Okounkov polytope} of $X(w)$ associated to $\mathcal{L}_\lambda$, $v$, and $\tau$.
\end{defi}\vspace{2mm}

We define a linear automorphism $\omega \colon \mathbb{R} \times \mathbb{R}^r \xrightarrow{\sim} \mathbb{R} \times \mathbb{R}^r$ by $\omega(k, {\bf a}) \coloneqq (k, -{\bf a})$. Recall that $\tau_\lambda = \Xi^{\rm up} _{\lambda, w} (b_\lambda) \in H^0(X(w), \mathcal{L}_\lambda)$. 

\vspace{2mm}\begin{prop}[{see \cite[Sect.\ 4]{Kav}}]\label{string polytopes}
Let ${\bf i} = (i_1, \ldots, i_r) \in I^r$ be a reduced word for $w \in W$, $\lambda \in P_+$, and ${\bf B}^{\rm up} = \{\Xi^{\rm up}(b) \mid b \in \mathcal{B}(\infty)\} \subset \c[U^-]$ the dual basis of a perfect basis.
\begin{enumerate}
\item[{\rm (1)}] Littelmann's string parametrization $\Phi_{\bf i} (b)$ is equal to $-v_{\bf i}(\Xi^{\rm up} _w (b))$ for every $b \in \mathcal{B}_w(\infty)$.
\item[{\rm (2)}] The following equalities hold$:$
\begin{align*}
&\mathcal{S}_{\bf i} ^{(\lambda, w)} = \omega(S(X(w), \mathcal{L}_\lambda, v_{\bf i}, \tau_\lambda)),\ \mathcal{C}_{\bf i} ^{(\lambda, w)} = \omega(C(X(w), \mathcal{L}_\lambda, v_{\bf i}, \tau_\lambda)),\ {\it and}\\
&\Delta_{\bf i} ^{(\lambda, w)} = -\Delta(X(w), \mathcal{L}_\lambda, v_{\bf i}, \tau_\lambda).
\end{align*}
\end{enumerate}
\end{prop}

\vspace{2mm}\begin{prop}[{see \cite[Sect.\ 4]{FN}}]\label{polyhedral realizations}
Let ${\bf i} = (i_1, \ldots, i_r) \in I^r$ be a reduced word for $w \in W$, $\lambda \in P_+$, and ${\bf B}^{\rm up} = \{\Xi^{\rm up}(b) \mid b \in \mathcal{B}(\infty)\} \subset \c[U^-]$ the dual basis of a perfect basis.
\begin{enumerate}
\item[{\rm (1)}] The Kashiwara embedding $\Psi_{\bf i} (b)$ is equal to $-\tilde{v}_{\bf i}(\Xi^{\rm up} _w (b))$ for every $b \in \mathcal{B}_w(\infty)$.
\item[{\rm (2)}] The following equalities hold$:$
\begin{align*}
&\widetilde{\mathcal{S}}_{\bf i} ^{(\lambda, w)} = \omega(S(X(w), \mathcal{L}_\lambda, \tilde{v}_{\bf i}, \tau_\lambda)),\ \widetilde{\mathcal{C}}_{\bf i} ^{(\lambda, w)} = \omega(C(X(w), \mathcal{L}_\lambda, \tilde{v}_{\bf i}, \tau_\lambda)),\ {\it and}\\ 
&\widetilde{\Delta}_{\bf i} ^{(\lambda, w)} = -\Delta(X(w), \mathcal{L}_\lambda, \tilde{v}_{\bf i}, \tau_\lambda).
\end{align*}
\end{enumerate}
\end{prop}\vspace{2mm}

\begin{rem}\normalfont
The author and Oya \cite{FO} proved that the valuations $v_{\bf i}, \tilde{v}_{\bf i}$ are also identical to ones given by counting the order of zeros along certain sequences of subvarieties of $X(w)$.
\end{rem}\vspace{2mm}

\section{Orbit Lie algebras}

In this section, we apply the folding procedure to crystal bases. First we recall from \cite{FRS, FSS} the definition of orbit Lie algebras. Recall that $\mathfrak{g}$ is assumed to be a finite-dimensional simple Lie algebra. We further assume that $\mathfrak{g}$ is of simply-laced type. Denote by $C = (c_{i, j})_{i, j \in I}$ the Cartan matrix of $\mathfrak{g}$, where $I$ is an index set for the vertices of the Dynkin diagram. Let $\omega \colon I \rightarrow I$ be a bijection of order $L$ satisfying $c_{\omega(i), \omega(j)} = c_{i, j}$ for all $i, j \in I$; such a bijection $\omega$ is called a {\it Dynkin diagram automorphism}. It induces a Lie algebra automorphism $\omega \colon \mathfrak{g} \xrightarrow{\sim} \mathfrak{g}$ of order $L$ defined by: \[\omega(E_i) = E_{\omega(i)},\ \omega(F_i) = F_{\omega(i)},\ \omega(h_i) = h_{\omega(i)}\] for $i \in I$; note that the Cartan subalgebra $\mathfrak{t}$ is invariant under $\omega$. Also, we define $\omega^\ast \colon \mathfrak{t}^\ast \xrightarrow{\sim} \mathfrak{t}^\ast$ by: $\omega^\ast(\lambda)(h) = \lambda(\omega^{-1}(h))$ for $\lambda \in \mathfrak{t}^\ast$ and $h \in \mathfrak{t}$. In this paper, we always impose the following orthogonality condition on $\omega$:
\begin{enumerate}
\item[{(O)}] $c_{i, j} = 0$\ {\rm for\ all}\ $i \neq j$\ {\rm in\ the\ same}\ $\omega$-{\rm orbit}.
\end{enumerate}
Let us fix a complete set $\breve{I} \subset I$ of representatives for the $\omega$-orbits in $I$. We set $m_i \coloneqq \min\{k \in \mathbb{Z}_{>0} \mid \omega^k(i) = i\}$ for $i \in I$, and then set \[\breve{c}_{i, j} \coloneqq \sum_{0 \le k < m_j} c_{i, \omega^k(j)}\] for $i, j \in \breve{I}$. Then we can verify that the matrix $\breve{C} \coloneqq (\breve{c}_{i, j})_{i, j \in \breve{I}}$ is an indecomposable Cartan matrix of finite type (see the list below). The finite-dimensional simple Lie algebra $\breve{\mathfrak{g}}$ with Cartan matrix $\breve{C}$ is called the {\it orbit Lie algebra} associated to $\omega$.
\begin{table}[h]
\begin{center}
\begin{tabular}{c|c} 
Dynkin diagram of $\mathfrak{g}$ & Dynkin diagram of $\breve{\mathfrak{g}}$\\\hline
\begin{xy}
\ar@{-} (20,4) *++!D{} *\cir<3pt>{};
(30,4) *++!D!R(0.4){} *\cir<3pt>{}="A",
\ar@{-} "A";(40,4) *++!D{} *\cir<3pt>{}="B"
\ar@{-} "B";(45,4) \ar@{.} (45,4);(50,4)^*!U{}
\ar@{-} (50,4);(55,4) *++!D{} *\cir<3pt>{}="C"
\ar@{-} "C";(64,0) *++!L{} *\cir<3pt>{}="D"
\ar@{-} "D";(55,-4) *++!D{} *\cir<3pt>{}="E"
\ar@{-} "E";(50,-4) \ar@{.} (50,-4);(45,-4)^*!U{}
\ar@{-} (45,-4);(40,-4) *++!D{} *\cir<3pt>{}="F"
\ar@{-} "F";(30,-4) *++!D{} *\cir<3pt>{}="G"
\ar@{-} "G";(20,-4) *++!D{} *\cir<3pt>{}
\end{xy}
 & \begin{xy}
\ar@{-} (50,0) *++!D{} *\cir<3pt>{};
(60,0) *++!D{} *\cir<3pt>{}="B"
\ar@{-} "B";(70,0) *++!D{} *\cir<3pt>{}="C"
\ar@{-} "C";(75,0) \ar@{.} (75,0);(80,0)^*!U{}
\ar@{-} (80,0);(85,0) *++!D{} *\cir<3pt>{}="D"
\ar@{=>} "D";(95,0) *++!D{} *\cir<3pt>{}="E"
\end{xy}\\
\begin{xy}
\ar@{-} (20,0) *++!D{} *\cir<3pt>{};
(30,0) *++!D!R(0.4){} *\cir<3pt>{}="A",
\ar@{-} "A";(40,0) *++!D{} *\cir<3pt>{}="B"
\ar@{-} "B";(45,0) \ar@{.} (45,0);(50,0)^*!U{}
\ar@{-} (50,0);(55,0) *++!D{} *\cir<3pt>{}="C"
\ar@{-} "C";(64,4) *++!L{} *\cir<3pt>{}
\ar@{-} "C";(64,-4) *++!L{} *\cir<3pt>{},
\end{xy} & \begin{xy}
\ar@{-} (50,0) *++!D{} *\cir<3pt>{};
(60,0) *++!D{} *\cir<3pt>{}="B"
\ar@{-} "B";(70,0) *++!D{} *\cir<3pt>{}="C"
\ar@{-} "C";(75,0) \ar@{.} (75,0);(80,0)^*!U{}
\ar@{-} (80,0);(85,0) *++!D{} *\cir<3pt>{}="D"
\ar@{<=} "D";(95,0) *++!D{} *\cir<3pt>{}="E"
\end{xy}\\
\begin{xy}
\ar@{-} (20,0) *++!D{} *\cir<3pt>{};
(30,0) *++!D!R(0.4){} *\cir<3pt>{}="C",
\ar@{-} "C";(39,4) *++!L{} *\cir<3pt>{}="D"
\ar@{-} "C";(39,-4) *++!L{} *\cir<3pt>{}="E",
\ar@{-} "D";(49,4) *++!L{} *\cir<3pt>{}
\ar@{-} "E";(49,-4) *++!L{} *\cir<3pt>{}
\end{xy} & \begin{xy}
\ar@{-} (50,0) *++!D{} *\cir<3pt>{};
(60,0) *++!D{} *\cir<3pt>{}="B"
\ar@{<=} "B";(70,0) *++!D{} *\cir<3pt>{}="C"
\ar@{-} "C";(80,0) *++!D{} *\cir<3pt>{}
\end{xy}\\
\begin{xy}
\ar@{-} (29,0) *++!D{} *\cir<3pt>{};
(20,0) *++!D!R(0.4){} *\cir<3pt>{}="C",
\ar@{-} "C";(29,4) *++!L{} *\cir<3pt>{}
\ar@{-} "C";(29,-4) *++!L{} *\cir<3pt>{}
\end{xy} & \begin{xy}
\ar@3{<-} *++!D{} *\cir<3pt>{};
(10,0) *++!D{} *\cir<3pt>{}
\end{xy}
\end{tabular}
\end{center}
\caption{The list of nontrivial Dynkin diagram automorphisms satisfying assumption (O).}\label{table1}
\end{table}

Let $U_q(\breve{\mathfrak g})$ be the quantized enveloping algebra of $\breve{\mathfrak g}$ with generators $\breve{e}_i, \breve{f}_i, \breve{t}_i, \breve{t}_i ^{-1}$, $i \in \breve{I}$, and $U_q(\breve{\mathfrak{u}}^-)$ the $\mathbb{Q}(q)$-subalgebra of $U_q(\breve{\mathfrak g})$ generated by $\{\breve{f}_i \mid i \in \breve{I}\}$. Denote by $\breve{\mathcal{B}}(\infty)$ the crystal basis of $U_q(\breve{\mathfrak{u}}^-)$, by $\breve{b}_\infty \in \breve{\mathcal{B}}(\infty)$ the element corresponding to $1 \in U_q(\breve{\mathfrak u}^-)$, and by $\tilde{e}_i, \tilde{f}_i \colon \breve{\mathcal{B}}(\infty) \cup \{0\} \rightarrow \breve{\mathcal{B}}(\infty) \cup \{0\}$, $i \in \breve{I}$, the Kashiwara operators. Then, the crystal basis $\breve{\mathcal{B}}(\infty)$ is realized as a specific subset of $\mathcal{B}(\infty)$; we recall this realization, following \cite{NS1, NS2, Sag}. The Dynkin diagram automorphism $\omega$ induces a $\mathbb{Q}(q)$-algebra automorphism $\omega \colon U_q(\mathfrak{g}) \xrightarrow{\sim} U_q(\mathfrak{g})$ of order $L$ defined by: \[\omega(e_i) = e_{\omega(i)},\ \omega(f_i) = f_{\omega(i)},\ \omega(t_i) = t_{\omega(i)}\] for $i \in I$; remark that $\omega$ preserves $U_q(\mathfrak{u}^-)$. We see from \cite[Sect.\ 3.4]{NS1} that this automorphism induces a natural bijection $\omega \colon \mathcal{B}(\infty) \rightarrow \mathcal{B}(\infty)$ such that 
\begin{equation}\label{omega negative part}
\omega \circ \tilde{e}_i = \tilde{e}_{\omega(i)} \circ \omega\ {\rm and}\ \omega \circ \tilde{f}_i = \tilde{f}_{\omega(i)} \circ \omega
\end{equation}
for all $i \in I$. Let us define operators $\tilde{e}^\omega _i, \tilde{f}^\omega _i \colon \mathcal{B}(\infty) \cup \{0\} \rightarrow \mathcal{B}(\infty) \cup \{0\}$ for $i \in I$ by:
\begin{equation}\label{omega Kashiwara operator}
\tilde{e}^\omega _i = \prod_{0 \le k < m_i} \tilde{e}_{\omega^k(i)}\ {\rm and}\ \tilde{f}^\omega _i = \prod_{0 \le k < m_i} \tilde{f}_{\omega^k(i)};
\end{equation}
note that the operators $\tilde{e}_{i}, \tilde{e}_{\omega(i)}, \ldots, \tilde{e}_{\omega^{m_i -1}(i)}$ (resp., $\tilde{f}_{i}, \tilde{f}_{\omega(i)}, \ldots, \tilde{f}_{\omega^{m_i -1}(i)}$) commute with each other by assumption (O); these operators $\tilde{e}^\omega _i, \tilde{f}^\omega _i$ are called the $\omega$-{\it Kashiwara operators}. Let $\breve{\mathfrak{t}} \subset \breve{\mathfrak{g}}$ be a Cartan subalgebra, $\{\breve{\alpha}_i \in \breve{\mathfrak{t}}^\ast \mid i \in \breve{I}\}$ the set of simple roots, $\{\breve{h}_i \in \breve{\mathfrak{t}} \mid i \in \breve{I}\}$ the set of simple coroots, and then set $\mathfrak{t}^0 \coloneqq \{h \in \mathfrak{t} \mid \omega(h) = h\}$ $(\mathfrak{t}^\ast)^0 \coloneqq \{\lambda \in \mathfrak{t}^\ast \mid \omega^\ast(\lambda) = \lambda\}$. As in \cite[Sect.\ 2]{FRS}, we obtain $\mathbb{C}$-linear isomorphisms $P_\omega \colon \mathfrak{t}^0 \xrightarrow{\sim} \breve{\mathfrak{t}}$ and $P_\omega ^\ast \colon \breve{\mathfrak{t}}^\ast \xrightarrow{\sim} (\mathfrak{t}^0)^\ast \simeq (\mathfrak{t}^\ast)^0$ such that 
\begin{align*}
P_\omega ^{-1}(\breve{h}_i) = \frac{1}{m_i} \sum_{0 \le k < m_i} h_{\omega^k(i)},\ P_\omega ^\ast(\breve{\alpha}_i) = \sum_{0 \le k < m_i} \alpha_{\omega^k(i)},\ {\rm and}\ (P_\omega ^\ast(\breve{\lambda}))(h) = \breve{\lambda}(P_\omega(h))
\end{align*}
for $i \in \breve{I}$, $\breve{\lambda} \in \breve{\mathfrak{t}}^\ast$, and $h \in \mathfrak{t}^0$. We denote by $\breve{W}$ the Weyl group of $\breve{\mathfrak g}$, and set \[\widetilde{W} \coloneqq \{w \in W \mid \omega^\ast \circ w = w \circ \omega^\ast\ {\rm on}\ \mathfrak{t}^\ast\}.\] Then we see from \cite[Sect.\ 3]{FRS} that there exists a group isomorphism $\Theta \colon \breve{W} \xrightarrow{\sim} \widetilde{W}$ such that $\Theta(\breve{w}) = P_\omega ^\ast \circ \breve{w} \circ (P_\omega ^\ast)^{-1}$ on $(\mathfrak{t}^\ast)^0$ for all $\breve{w} \in \breve{W}$.

\vspace{2mm}\begin{prop}[{\cite[Theorem 3.4.1]{NS1}}]\label{orbit infinity}
Let \[\mathcal{B}^0(\infty) \coloneqq \{b \in \mathcal{B}(\infty) \mid \omega(b) = b\}\] denote the fixed point subset by $\omega$.
\begin{enumerate}
\item[{\rm (1)}] The set $\mathcal{B}^0(\infty) \cup \{0\}$ is stable under the $\omega$-Kashiwara operators $\tilde{e}_i ^\omega, \tilde{f}_i ^\omega$ for all $i \in I$. 
\item[{\rm (2)}] There exists a unique bijective map $P_\infty \colon \mathcal{B}^0(\infty) \cup \{0\} \rightarrow \breve{\mathcal{B}}(\infty) \cup \{0\}$ such that 
\begin{align*}
P_\infty(b_\infty) = \breve{b}_\infty,\ P_\infty \circ \tilde{e}_i ^\omega = \tilde{e}_i \circ P_\infty,\ {\it and}\ P_\infty \circ \tilde{f}_i ^\omega = \tilde{f}_i \circ P_\infty
\end{align*}
for all $i \in \breve{I}$.
\item[{\rm (3)}] The equality \[P_\infty(\mathcal{B}_{\Theta(w)} ^0(\infty)) = \breve{\mathcal{B}}_w(\infty)\] holds for every $w \in \breve{W}$, where $\mathcal{B}_{\Theta(w)} ^0(\infty) \coloneqq \mathcal{B}^0(\infty) \cap \mathcal{B}_{\Theta(w)}(\infty)$.
\end{enumerate}
\end{prop}\vspace{2mm}

For $i \in \breve{I}$ and $b \in \mathcal{B}^0(\infty)$, we set \[\varepsilon_i ^\omega(b) \coloneqq \max\{a \in \mathbb{Z}_{\ge 0} \mid (\tilde{e}^\omega _i)^a b \neq 0\}.\] The properties of $P_\infty$ in Proposition \ref{orbit infinity} (2) imply the equality \[\varepsilon_i ^\omega(b) = \varepsilon_i (P_\infty(b))\] for every $i \in \breve{I}$ and $b \in \mathcal{B}^0(\infty)$. 

\vspace{2mm}\begin{prop}\label{folding of epsilon}
The equality \[\varepsilon_i ^\omega(b) = \varepsilon_{\omega^k(i)}(b)\] holds for every $i \in \breve{I}$, $k \in \mathbb{Z}_{\ge0}$, and $b \in \mathcal{B}^0(\infty)$.
\end{prop}

\begin{proof}
Although this is proved in \cite[Lemma 2.3.2]{NS2}, we give a proof for the convenience of the reader. By replacing $\breve{I}$ if necessary, we may assume that $k = 0$. Since $(\tilde{e}^\omega _i)^a = \tilde{e}_{\omega^{m_i -1}(i)} ^a \cdots \tilde{e}_{\omega(i)} ^a \tilde{e}_i ^a$ for $a \in \z_{\ge 0}$ by assumption (O), the condition $(\tilde{e}^\omega _i)^{\varepsilon_i ^\omega(b)}b \neq 0$ implies that $\tilde{e}_i^{\varepsilon_i ^\omega(b)}b \neq 0$. Suppose, for a contradiction, that $\tilde{e}_i^{\varepsilon_i ^\omega(b)+1}b \neq 0$. Then we have 
\begin{align*}
\tilde{e}_{\omega^k(i)} ^{\varepsilon_i ^\omega(b)+1}b &= \tilde{e}_{\omega^k(i)} ^{\varepsilon_i ^\omega(b)+1} \omega^k(b)\quad({\rm since}\ b \in \mathcal{B}^0(\infty))\\
&= \omega^k(\tilde{e}_i ^{\varepsilon_i ^\omega(b)+1} b)\quad({\rm by\ equation}\ (\ref{omega negative part}))\\
&\neq 0,
\end{align*}
from which we deduce by assumption (O) that \[\tilde{e}_{\omega^k(i)} ^{\varepsilon_i ^\omega(b)+1} \cdots \tilde{e}_{\omega(i)} ^{\varepsilon_i ^\omega(b)+1} \tilde{e}_i ^{\varepsilon_i ^\omega(b)+1} b \neq 0\] for any $0 \le k \le m_i -1$; this contradicts the equality $(\tilde{e}_i ^\omega)^{\varepsilon_i ^\omega(b)+1} b = 0$. Therefore, the equality $\tilde{e}_i^{\varepsilon_i ^\omega(b)+1}b = 0$ holds, which implies that $\varepsilon_i (b) = \varepsilon_i ^\omega(b)$. This proves the proposition.
\end{proof}

Note that $\breve{P} \coloneqq (P_\omega ^\ast)^{-1}(P \cap (\mathfrak{t}^\ast)^0) \subset \breve{\mathfrak{t}}^\ast$ is identical to the weight lattice for $\breve{\mathfrak{g}}$. For $\lambda \in P_+ \cap (\mathfrak{t}^\ast)^0$, we have a natural bijective map $\omega \colon \mathcal{B}(\lambda) \rightarrow \mathcal{B}(\lambda)$, induced by the $\mathbb{Q}(q)$-algebra automorphism $\omega \colon U_q(\mathfrak{g}) \xrightarrow{\sim} U_q(\mathfrak{g})$, such that
\begin{equation}\label{omega highest weight module}
\omega \circ \tilde{e}_i = \tilde{e}_{\omega(i)} \circ \omega\ {\rm and}\ \omega \circ \tilde{f}_i = \tilde{f}_{\omega(i)} \circ \omega
\end{equation}
for all $i \in I$ (see \cite[Sect.\ 3.2]{NS1} and \cite[Sect.\ 3]{Sag}). Here we recall that $\pi_\lambda \colon \mathcal{B}(\infty) \twoheadrightarrow \mathcal{B}(\lambda) \cup \{0\}$ is the canonical map induced from the natural surjection $U_q(\mathfrak{u}^-) \twoheadrightarrow V_q(\lambda)$. If we set \[\mathcal{B}^0(\lambda) \coloneqq \{b \in \mathcal{B}(\lambda) \mid \omega(b) = b\},\] then it is easily checked that $\pi_\lambda(\mathcal{B}^0(\infty)) = \mathcal{B}^0(\lambda) \cup \{0\}$. For $\breve{\lambda} \in (P_\omega ^\ast)^{-1}(P_+ \cap (\mathfrak{t}^\ast)^0)$, let $\breve{V}_q(\breve{\lambda})$ denote the irreducible highest weight $U_q(\breve{\mathfrak g})$-module with highest weight $\breve{\lambda}$, $\breve{\mathcal{B}}(\breve{\lambda})$ the crystal basis of $\breve{V}_q(\breve{\lambda})$ with $b_{\breve{\lambda}} \in \breve{\mathcal{B}}(\breve{\lambda})$ the highest element, and $\tilde{e}_i, \tilde{f}_i \colon \breve{\mathcal{B}}(\breve{\lambda}) \cup \{0\} \rightarrow \breve{\mathcal{B}}(\breve{\lambda}) \cup \{0\}$, $i \in \breve{I}$, the Kashiwara operators. 

\vspace{2mm}\begin{prop}[{\cite[Proposition 3.2.1]{NS1}}]\label{compatibility with lambda}
Let $\lambda \in P_+ \cap (\mathfrak{t}^\ast)^0$.
\begin{enumerate}
\item[{\rm (1)}] The set $\mathcal{B}^0(\lambda) \cup \{0\}$ is stable under the $\omega$-Kashiwara operators $\tilde{e}_i ^\omega, \tilde{f}_i ^\omega \colon \mathcal{B}(\lambda) \cup \{0\} \rightarrow \mathcal{B}(\lambda) \cup \{0\}$ for all $i \in I$, defined in the same way as $\omega$-Kashiwara operators for $\mathcal{B}(\infty)$.
\item[{\rm (2)}] There exists a unique bijective map $P_\lambda \colon \mathcal{B}^0(\lambda) \cup \{0\} \rightarrow \breve{\mathcal{B}}(\breve{\lambda}) \cup \{0\}$ such that 
\begin{align*}
P_\lambda(b_\lambda) = b_{\breve{\lambda}},\ P_\lambda \circ \tilde{e}_i ^\omega = \tilde{e}_i \circ P_\lambda\ {\it and}\ P_\lambda \circ \tilde{f}_i ^\omega = \tilde{f}_i \circ P_\lambda
\end{align*}
for all $i \in \breve{I}$, where $\breve{\lambda} \coloneqq (P_\omega ^\ast)^{-1}(\lambda)$.
\item[{\rm (3)}] The following diagram is commutative$:$
\begin{align*}
\xymatrix{\mathcal{B}^0(\infty) \ar[d]_-{P_\infty} \ar[r]^-{\pi_\lambda} & \mathcal{B}^0(\lambda) \cup \{0\} \ar[d]_-{P_\lambda} \\
\breve{\mathcal{B}}(\infty) \ar[r]^-{\pi_{\breve{\lambda}}} & \breve{\mathcal{B}}(\breve{\lambda}) \cup \{0\},}
\end{align*}
where $\pi_{\breve{\lambda}}$ is the map induced from the natural surjective map $U_q(\breve{\mathfrak{u}}^-) \twoheadrightarrow \breve{V}_q(\breve{\lambda})$.
\item[{\rm (4)}] The equality \[P_\lambda(\mathcal{B}_{\Theta(w)} ^0(\lambda)) = \breve{\mathcal{B}}_{w}(\breve{\lambda})\] holds for all $w \in \breve{W}$, where $\mathcal{B}_{\Theta(w)} ^0(\lambda) \coloneqq \mathcal{B}^0(\lambda) \cap \mathcal{B}_{\Theta(w)}(\lambda)$ and $\breve{\mathcal{B}}_{w}(\breve{\lambda}) \subset \breve{\mathcal{B}}(\breve{\lambda})$ is the corresponding Demazure crystal. 
\end{enumerate}
\end{prop}\vspace{2mm}

\begin{rem}\normalfont
The composite maps $\breve{\mathcal{B}}(\infty) \xrightarrow{P_\infty ^{-1}} \mathcal{B}^0(\infty) \hookrightarrow \mathcal{B}(\infty)$ and $\breve{\mathcal{B}}(\breve{\lambda}) \xrightarrow{P_\lambda ^{-1}} \mathcal{B}^0(\lambda) \hookrightarrow \mathcal{B}(\lambda)$ are identical to the maps arising from a similarity of crystal bases (see \cite[Sect.\ 5]{Kas5}). This similarity is a variant of what we consider in Section 6.
\end{rem}\vspace{2mm}

It is easily seen that $\omega \circ \ast = \ast \circ \omega$ on $U_q(\mathfrak{g})$, which implies the same equality on $\mathcal{B}(\infty)$. Hence it follows that $\mathcal{B}^0(\infty)^\ast = \mathcal{B}^0(\infty)$. We denote by $\ast \colon \breve{\mathcal{B}}(\infty) \rightarrow \breve{\mathcal{B}}(\infty)$ Kashiwara's involution on $\breve{\mathcal{B}}(\infty)$.

\vspace{2mm}\begin{prop}[{\cite[Theorem1]{NS2}}]\label{compatible with Kashiwara's involution}
The following diagram is commutative$:$
\begin{align*}
\xymatrix{\mathcal{B}^0(\infty) \ar[d]_-{P_\infty} \ar[r]^-{\ast} & \mathcal{B}^0(\infty) \ar[d]_-{P_\infty} \\
\breve{\mathcal{B}}(\infty) \ar[r]^-{\ast} & \breve{\mathcal{B}}(\infty).}
\end{align*}
\end{prop}\vspace{2mm}

The following is an immediate consequence of Propositions \ref{folding of epsilon} and \ref{compatible with Kashiwara's involution}.

\vspace{2mm}\begin{cor}\label{folding of epsilon ast}
The equality \[\varepsilon_i(P_\infty(b)^\ast) = \varepsilon_{\omega^k(i)} (b^\ast)\] holds for all $i \in \breve{I}$, $k \in \mathbb{Z}_{\ge 0}$, and $b \in \mathcal{B}^0(\infty)$.
\end{cor}\vspace{2mm}

Let $\{s_i \mid i \in I\} \subset W$ (resp., $\{s_i \mid i \in \breve{I}\} \subset \breve{W}$) be the set of simple reflections. If we take a reduced word ${\bf i} = (i_1, \ldots, i_r) \in \breve{I}^r$ for $w \in \breve{W}$, then we have \[\Theta(w) = \Theta(s_{i_1}) \cdots \Theta(s_{i_r}) = s_{i_{1, 1}} \cdots s_{i_{1, m_{i_1}}} \cdots s_{i_{r, 1}} \cdots s_{i_{r, m_{i_r}}},\] where we set $i_{k, l} \coloneqq \omega^{l-1}(i_k)$ for $1 \le k \le r$ and $1 \le l \le m_{i_k}$. It is easily verified that this is a reduced expression for $\Theta(w)$; we denote by $\Theta({\bf i})$ the corresponding reduced word $(i_{1, 1}, \ldots, i_{1, m_{i_1}}, \ldots, i_{r, 1}, \ldots, i_{r, m_{i_r}})$.

\vspace{2mm}\begin{cor}\label{Key corollary 1}
Let ${\bf i} = (i_1, \ldots, i_r) \in \breve{I}^r$ be a reduced word for $w \in \breve{W}$. Define an $\mathbb{R}$-linear injective map $\Upsilon_{\bf i} \colon \mathbb{R}^r \hookrightarrow \mathbb{R}^{m_{i_1} + \cdots + m_{i_r}}$ by$:$ \[\Upsilon_{\bf i}(a_1, \ldots, a_r) = (\underbrace{a_1, \ldots, a_1}_{m_{i_1}}, \ldots, \underbrace{a_r, \ldots, a_r}_{m_{i_r}}).\] Then, the equalities \[\Upsilon_{\bf i}(\Phi_{\bf i}(b)) = \Phi_{\Theta({\bf i})}(P_\infty ^{-1}(b))\ {\it and}\ \Upsilon_{\bf i}(\Psi_{\bf i}(b)) = \Psi_{\Theta({\bf i})}(P_\infty ^{-1}(b))\] hold for all $b \in \breve{\mathcal{B}}_w(\infty)$. In particular, the following equalities hold$:$ \[\Upsilon_{\bf i}(\Phi_{\bf i}(\breve{\mathcal{B}}_w(\infty))) = \Phi_{\Theta({\bf i})}(\mathcal{B}^0 _{\Theta(w)}(\infty)),\ {\it and}\ \Upsilon_{\bf i}(\Psi_{\bf i}(\breve{\mathcal{B}}_w(\infty))) = \Psi_{\Theta({\bf i})}(\mathcal{B}^0 _{\Theta(w)}(\infty)).\]
\end{cor}

\begin{proof}
We take $b \in \breve{\mathcal{B}}_w(\infty)$, and write $\Phi_{\bf i} (b)$ as $(a_1, \ldots, a_r)$. We will show that \[\Phi_{\Theta({\bf i})}(P_\infty ^{-1}(b)) = (\underbrace{a_1, \ldots, a_1}_{m_{i_1}}, \ldots, \underbrace{a_r, \ldots, a_r}_{m_{i_r}}).\] It follows by assumption (O) and Proposition \ref{folding of epsilon} that \[\varepsilon_{i_{1, k}}(\tilde{e}_{i_{1, k-1}} ^{a_1} \cdots \tilde{e}_{i_{1, 1}} ^{a_1} b) = \varepsilon_{i_{1, k}}(b) = a_1\] for all $1 \le k \le m_{i_1}$ (see also the proof of Proposition \ref{folding of epsilon}). Therefore, the following equality holds: \[\Phi_{\Theta({\bf i})}(P_\infty ^{-1}(b)) = (\underbrace{a_1, \ldots, a_1}_{m_{i_1}}, \Phi_{\Theta({\bf i}_{\ge 2})}(P_\infty ^{-1}(b'))),\] where ${\bf i}_{\ge 2} \coloneqq (i_2, \ldots, i_r)$ and $b' \coloneqq \tilde{e}_{i_{1, m_{i_1}}} ^{a_1} \cdots \tilde{e}_{i_{1, 1}} ^{a_1} b$. Moreover, by induction on $r$, we deduce that \[\Phi_{\Theta({\bf i}_{\ge 2})}(P_\infty ^{-1}(b')) = (\underbrace{a_2, \ldots, a_2}_{m_{i_2}}, \ldots, \underbrace{a_r, \ldots, a_r}_{m_{i_r}}).\] From these, we obtain the assertion for $\Phi_{\bf i}$. The assertion for $\Psi_{\bf i}$ is shown similarly by using Corollary \ref{folding of epsilon ast} instead of Proposition \ref{folding of epsilon}.
\end{proof}

If $b \in \mathcal{B}_{\Theta(w)}(\infty)$ satisfies $\Phi_{\Theta({\bf i})}(b) = \Upsilon_{\bf i}(a_1, \ldots, a_r)$ for some $(a_1, \ldots, a_r) \in \z_{\ge 0} ^r$, then it is easily seen that $b \in \mathcal{B}^0 _{\Theta(w)}(\infty)$. Hence we obtain the following.

\vspace{2mm}\begin{cor}\label{corollary slice upsilon}
Let ${\bf i} = (i_1, \ldots, i_r) \in \breve{I}^r$ be a reduced word for $w \in \breve{W}$. Then the following equalities hold$:$
\begin{align*}
&\Upsilon_{\bf i}(\Phi_{\bf i}(\breve{\mathcal{B}}_{w}(\infty))) = \{(a_{k, l})_{1 \le k \le r, 1\le l \le m_{i_k}} \in \Phi_{\Theta({\bf i})}(\mathcal{B}_{\Theta(w)}(\infty)) \mid a_{k, 1} = \cdots = a_{k, m_{i_k}},\ 1 \le k \le r\},\ {\it and}\\
&\Upsilon_{\bf i}(\Psi_{\bf i}(\breve{\mathcal{B}}_{w}(\infty))) = \{(a_{k, l})_{1 \le k \le r, 1\le l \le m_{i_k}} \in \Psi_{\Theta({\bf i})}(\mathcal{B}_{\Theta(w)}(\infty)) \mid a_{k, 1} = \cdots = a_{k, m_{i_k}},\ 1 \le k \le r\}.
\end{align*}
\end{cor}\vspace{2mm}

Similarly, we obtain the following (see Proposition \ref{compatibility with lambda} (3), (4)).

\vspace{2mm}\begin{cor}\label{polytopes as slices}
Let ${\bf i} = (i_1, \ldots, i_r) \in \breve{I}^r$ be a reduced word for $w \in \breve{W}$, and $\lambda \in P_+ \cap (\mathfrak{t}^\ast)^0$. Then the following equalities hold$:$ 
\begin{align*}
&\Upsilon_{\bf i}(\Phi_{\bf i}(\breve{\mathcal{B}}_{w}(\breve{\lambda}))) = \{(a_{k, l})_{1 \le k \le r, 1\le l \le m_{i_k}} \in \Phi_{\Theta({\bf i})}(\mathcal{B}_{\Theta(w)}(\lambda)) \mid a_{k, 1} = \cdots = a_{k, m_{i_k}},\ 1 \le k \le r\},\ {\it and}\\
&\Upsilon_{\bf i}(\Psi_{\bf i}(\breve{\mathcal{B}}_{w}(\breve{\lambda}))) = \{(a_{k, l})_{1 \le k \le r, 1\le l \le m_{i_k}} \in \Psi_{\Theta({\bf i})}(\mathcal{B}_{\Theta(w)}(\lambda)) \mid a_{k, 1} = \cdots = a_{k, m_{i_k}},\ 1 \le k \le r\},
\end{align*}
where $\breve{\lambda} \coloneqq (P_\omega ^\ast)^{-1}(\lambda)$.
\end{cor}\vspace{2mm}

By the definitions of Littelmann's string polytopes and Nakashima-Zelevinsky's polytopes, we obtain the following as an immediate consequence of Corollary \ref{polytopes as slices}.

\vspace{2mm}\begin{cor}\label{c:folding of polytopes(crystal)}
Let ${\bf i} = (i_1, \ldots, i_r) \in \breve{I}^r$ be a reduced word for $w \in \breve{W}$, and $\lambda \in P_+ \cap (\mathfrak{t}^\ast)^0$. Then the following equalities hold$:$
\begin{align*}
&\Upsilon_{\bf i}(\Delta_{\bf i} ^{(\breve{\lambda}, w)}) = \{(a_{k, l})_{1 \le k \le r, 1\le l \le m_{i_k}} \in \Delta_{\Theta({\bf i})} ^{(\lambda, \Theta(w))} \mid a_{k, 1} = \cdots = a_{k, m_{i_k}},\ 1 \le k \le r\},\ {\it and}\\
&\Upsilon_{\bf i}(\widetilde{\Delta}_{\bf i} ^{(\breve{\lambda}, w)}) = \{(a_{k, l})_{1 \le k \le r, 1\le l \le m_{i_k}} \in \widetilde{\Delta}_{\Theta({\bf i})} ^{(\lambda, \Theta(w))} \mid a_{k, 1} = \cdots = a_{k, m_{i_k}},\ 1 \le k \le r\},
\end{align*}
where $\breve{\lambda} \coloneqq (P_\omega ^\ast)^{-1}(\lambda)$.
\end{cor}

\vspace{2mm}\begin{rem}\normalfont
Corollary \ref{c:folding of polytopes(crystal)} is naturally extended to string polytopes for generalized Demazure modules, defined in \cite{Fuj}.
\end{rem}\vspace{2mm}

\section{Fixed point Lie subalgebras}

In this section, we prove our main result. Let us consider the fixed point Lie subalgebra by $\omega$ \[\mathfrak{g}^\omega \coloneqq \{x \in \mathfrak{g} \mid \omega(x) = x\}.\] Define $E_i ^\prime, F_i ^\prime, h_i ^\prime \in \mathfrak{g}^\omega$ and $\alpha^\prime _i \in (\mathfrak{t}^\ast)^0$ for $i \in \breve{I}$ by \[E_i ^\prime \coloneqq \sum_{0 \le k < m_i} E_{\omega^k(i)},\ F_i ^\prime \coloneqq \sum_{0 \le k < m_i} F_{\omega^k(i)},\ h_i ^\prime \coloneqq \sum_{0 \le k < m_i} h_{\omega^k(i)}\ {\rm and}\ \alpha_i ^\prime \coloneqq \frac{1}{m_i} \sum_{0 \le k < m_i} \alpha_{\omega^k(i)}.\] We set $c^\prime _{i, j} \coloneqq \langle \alpha_j ^\prime, h_i ^\prime \rangle$ for $i, j \in \breve{I}$. Then, it is easily checked that $\breve{c}_{i, j} = c^\prime _{j, i}$ for all $i, j \in \breve{I}$; namely, the matrix $C^\prime \coloneqq (c^\prime _{i, j})_{i, j \in \breve{I}}$ is the transpose of $\breve{C}$. In particular, the matrix $C^\prime$ is an indecomposable Cartan matrix of finite type.

\vspace{2mm}\begin{prop}[{see \cite[Proposition 8.3]{Kac}}]
The fixed point Lie subalgebra ${\mathfrak g}^\omega$ is the simple Lie algebra with Cartan matrix $C^\prime$ and Chevalley generators $E_i ^\prime, F_i ^\prime, h_i ^\prime$, $i \in \breve{I};$ in particular, the orbit Lie algebra $\breve{\mathfrak g}$ associated to $\omega$ is the (Langlands) dual Lie algebra of ${\mathfrak g}^\omega$. 
\end{prop}\vspace{2mm}

Recall that $G$ is the connected, simply-connected simple algebraic group with ${\rm Lie}(G) = \mathfrak{g}$. The Lie algebra automorphism $\omega \colon \mathfrak{g} \xrightarrow{\sim} \mathfrak{g}$ induces an algebraic group automorphism $\omega \colon G \xrightarrow{\sim} G$ such that $\omega(\exp(x)) = \exp(\omega(x))$ for all $x \in \mathfrak{g}$. It is known that the fixed point subgroup \[G^\omega \coloneqq \{g \in G \mid \omega(g) = g\}\] is a connected simple algebraic group with ${\rm Lie}(G^\omega) = \mathfrak{g}^\omega$; note that $G^\omega$ is a Zariski closed subgroup of $G$. In addition, we see by Table 1 in Section 4 and a case-by-case argument that $G^\omega$ is simply-connected. Since the fixed point subgroup $(U^-)^\omega \coloneqq U^- \cap G^\omega$ is a Zariski closed subgroup of $U^-$, the coordinate ring $\mathbb{C}[(U^-)^\omega]$ is a quotient of $\mathbb{C}[U^-]$; denote by $\pi^\omega \colon \mathbb{C}[U^-] \twoheadrightarrow \mathbb{C}[(U^-)^\omega]$ the quotient map. We set $B^\omega \coloneqq B \cap G^\omega$, and consider the full flag variety $G^\omega/B^\omega$. Let $\iota^\omega \colon G^\omega/B^\omega \hookrightarrow G/B$ denote the natural injective map. Since $\omega(B) = B$, the automorphism $\omega \colon G \xrightarrow{\sim} G$ induces a variety automorphism $\omega \colon G/B \xrightarrow{\sim} G/B$, and the image of $\iota^\omega$ is identical to the fixed point subvariety $(G/B)^\omega$. In addition, the map $\iota^\omega$ induces a $\c$-linear isomorphism from the tangent space of $G^\omega/B^\omega$ at $e \bmod B^\omega$ to that of $(G/B)^\omega$ at $e \bmod B$, where $e \in G^\omega$ ($\subset G$) is the identity element; note that both of these tangent spaces are identified with the Lie subalgebra of $\mathfrak{g}^\omega$ generated by $\{F_i ^\prime \mid i \in \breve{I}\}$. Therefore, the map $\iota^\omega \colon G^\omega/B^\omega \rightarrow (G/B)^\omega$ is an isomorphism of varieties (see, for instance, \cite[Sect.\ 5]{Spr}). Here we note that since $\mathfrak{g}^\omega$ is the (Langlands) dual Lie algebra of $\breve{\mathfrak{g}}$, the Weyl group $\breve{W}$ of $\breve{\mathfrak{g}}$ is identified with that of $\mathfrak{g}^\omega$. We consider the Schubert variety $X(w) \subset G^\omega/B^\omega \simeq (G/B)^\omega$ corresponding to $w \in \breve{W}$; this is identified with a Zariski closed subvariety of $X(\Theta(w))$. Let us regard $(U^-)^\omega$ as an affine open subvariety of $G^\omega/B^\omega$, and take the intersection $(U^-)^\omega \cap X(w)$ in $G^\omega/B^\omega$ for $w \in \breve{W}$; this intersection is identified with a Zariski closed subvariety of $U^- \cap X(\Theta(w))$. Let $\pi^\omega _w \colon \mathbb{C}[U^- \cap X(\Theta(w))] \twoheadrightarrow \mathbb{C}[(U^-)^\omega \cap X(w)]$ be the restriction map for $w \in \breve{W}$. We take a reduced word ${\bf i} = (i_1, \ldots, i_r) \in \breve{I}^r$ for $w \in \breve{W}$, and regard the coordinate ring $\mathbb{C}[(U^-)^\omega \cap X(w)]$ as a $\mathbb{C}$-subalgebra of the polynomial ring $\mathbb{C}[t_1, \ldots, t_r]$ by the following birational morphism: 
\[\c^r \rightarrow (U^-)^\omega \cap X(w),\ (t_1, \ldots, t_r) \mapsto \exp(t_1 F_{i_1} ^\prime) \cdots \exp(t_r F_{i_r} ^\prime).\] 
Since $\Theta({\bf i}) = (i_{1, 1}, \ldots, i_{1, m_{i_1}}, \ldots, i_{r, 1}, \ldots, i_{r, m_{i_r}})$ is a reduced word for $\Theta(w) \in W$, the coordinate ring $\mathbb{C}[U^- \cap X(\Theta(w))]$ is regarded as a $\mathbb{C}$-subalgebra of the polynomial ring $\mathbb{C}[t_{k, l} \mid 1 \le k \le r,\ 1 \le l \le m_{i_k}]$ by the following birational morphism: 
\begin{align*}
\c^{m_{i_1} + \cdots + m_{i_r}} \rightarrow U^- \cap X(\Theta(w)),\ (t_{1, 1}, \ldots, t_{r, m_{i_r}}) \mapsto \exp(t_{1, 1} F_{i_{1, 1}}) \cdots \exp(t_{r, m_{i_r}} F_{i_{r, m_{i_r}}}).
\end{align*}
Also, under the inclusion map $(U^-)^\omega \cap X(w) \hookrightarrow U^- \cap X(\Theta(w))$, we have \[\exp(tF_{i_k} ^\prime) \mapsto \exp(tF_{i_{k, 1}}) \cdots \exp(tF_{i_{k, m_{i_k}}})\] for $t \in \c$ and $1 \le k \le r$. Hence we obtain the following.

\vspace{2mm}\begin{lem}\label{polynomial folding}
Define a surjective map $\pi^\omega _{\bf i} \colon \mathbb{C}[t_{k, l} \mid 1 \le k \le r,\ 1 \le l \le m_{i_k}] \twoheadrightarrow \mathbb{C}[t_1, \ldots, t_r]$ by $\pi^\omega _{\bf i}(t_{k, l}) \coloneqq t_k$ for $1 \le k \le r$ and $1 \le l \le m_{i_k}$. Then the following diagram is commutative$:$
\begin{align*}
\xymatrix{\mathbb{C}[U^- \cap X(\Theta(w))] \ar@{^{(}->}[r] \ar[d]^-{\pi^\omega _w} & \mathbb{C}[t_{k, l} \mid 1 \le k \le r,\ 1 \le l \le m_{i_k}] \ar[d]^-{\pi^\omega _{\bf i}}\\
\mathbb{C}[(U^-)^\omega \cap X(w)] \ar@{^{(}->}[r] & \mathbb{C}[t_1, \ldots, t_r].}
\end{align*}
\end{lem}\vspace{2mm}

\begin{defi}\normalfont
Define a $\mathbb{C}$-algebra homomorphism $\Delta \colon U(\mathfrak{u}^-) \rightarrow U(\mathfrak{u}^-) \otimes U(\mathfrak{u}^-)$ by $\Delta(x) = x \otimes 1 + 1 \otimes x$ for $x \in \mathfrak{u}^-$.
\end{defi}\vspace{2mm}

Let us consider a perfect basis ${\bf B}^{\rm low} = \{\Xi^{\rm low} (b) \mid b \in \mathcal{B}(\infty)\} \subset U(\mathfrak{u}^-)$ that satisfies the following positivity conditions:
\begin{enumerate}
\item[${\rm (P)}_1$] the element $F_i \cdot \Xi^{\rm low} (b)$ belongs to $\sum_{b^\prime \in \mathcal{B}(\infty)} \mathbb{R}_{\ge 0} \Xi^{\rm low} (b^\prime)$ for every $b \in \mathcal{B}(\infty)$ and $i \in I$;
\item[${\rm (P)}_2$] the element $\Delta(\Xi^{\rm low} (b))$ belongs to $\sum_{b^\prime, b'' \in \mathcal{B}(\infty)} \mathbb{R}_{\ge 0} \Xi^{\rm low} (b^\prime) \otimes \Xi^{\rm low} (b'')$ for every $b \in \mathcal{B}(\infty)$. 
\end{enumerate}

\vspace{2mm}\begin{rem}\normalfont
In the paper \cite{FO}, the author and Oya used a perfect basis that satisfies slightly weaker positivity conditions; in it, positivity conditions are imposed only on certain coefficients of $\Delta(\Xi^{\rm low} (b))$. 
\end{rem}

\vspace{2mm}\begin{ex}\normalfont
Recall that $\mathfrak{g}$ is of simply-laced type. In this case, Lusztig proved that the specialization of the lower global basis at $q=1$ satisfies positivity conditions ${\rm (P)}_1, {\rm (P)}_2$ by using the geometric construction of the lower global basis \cite[Theorem 11.5]{Lus_quivers}. 
\end{ex}\vspace{2mm}

\begin{lem}\label{l:positivity lemma}
Let ${\bf B}^{\rm low} = \{\Xi^{\rm low} (b) \mid b \in \mathcal{B}(\infty)\} \subset U(\mathfrak{u}^-)$ be a perfect basis satisfying ${\rm (P)}_1, {\rm (P)}_2$. 
\begin{enumerate}
\item[{\rm (1)}] The perfect basis ${\bf B}^{\rm low}$ satisfies condition {\rm (D)} in Section {\rm 3}.
\item[{\rm (2)}] The element $\Xi_{\Theta(w)} ^{\rm up}(b) \cdot \Xi_{\Theta(w)} ^{\rm up}(b^\prime)$ belongs to $\sum_{b'' \in \mathcal{B}_{\Theta(w)}(\infty)} \mathbb{R}_{\ge 0} \Xi_{\Theta(w)} ^{\rm up}(b'')$ for all $w \in \breve{W}$ and $b, b^\prime \in \mathcal{B}_{\Theta(w)}(\infty);$ in addition, the coefficient of $\Xi_{\Theta(w)} ^{\rm up}(b'')$ is not equal to $0$ if $\Phi_{\Theta({\bf i})}(b'') = \Phi_{\Theta({\bf i})}(b) + \Phi_{\Theta({\bf i})}(b^\prime)$ or if $\Psi_{\Theta({\bf i})}(b'') = \Psi_{\Theta({\bf i})}(b) + \Psi_{\Theta({\bf i})}(b^\prime)$.
\item[{\rm (3)}] The coefficient of $t_{1, 1} ^{a_{1, 1}} \cdots t_{r, m_{i_r}} ^{a_{r, m_{i_r}}}$ in $\Xi_{\Theta(w)} ^{\rm up}(b) \in \mathbb{C}[t_{k, l} \mid 1 \le k \le r,\ 1 \le l \le m_{i_k}]$ is a nonnegative real number for all $w \in \breve{W}$, $b \in \mathcal{B}_{\Theta(w)}(\infty)$, and $a_{1, 1}, \ldots, a_{r, m_{i_r}} \in \mathbb{Z}_{\ge 0}$.
\end{enumerate}
\end{lem}

\begin{proof}
Parts (1), (3), and the first assertion of part (2) are proved in a way similar to \cite[Propositions 4.3, 4.7 and Corollary 4.6 (2)]{FO}. The second assertion of part (2) follows from general properties of valuations (see \cite[Sect.\ 6]{Kav}).
\end{proof}

\begin{thm}\label{Key proposition 1}
Let ${\bf i} = (i_1, \ldots, i_r) \in \breve{I}^r$ be a reduced word for $w \in \breve{W}$, and ${\bf B}^{\rm low} = \{\Xi^{\rm low} (b) \mid b \in \mathcal{B}(\infty)\} \subset U(\mathfrak{u}^-)$ a perfect basis satisfying ${\rm (P)}_1, {\rm (P)}_2$. Define an $\mathbb{R}$-linear surjective map $\Omega_{\bf i} \colon \mathbb{R}^{m_{i_1} + \cdots + m_{i_r}} \twoheadrightarrow \mathbb{R}^r$ by$:$ \[\Omega_{\bf i}(a_{1, 1}, \ldots, a_{1, m_{i_1}}, \ldots, a_{r, 1}, \ldots, a_{r, m_{i_r}}) = (a_{1, 1} + \cdots + a_{1, m_{i_1}}, \ldots, a_{r, 1} + \cdots + a_{r, m_{i_r}}).\] Then the following equalities hold for all $b \in \mathcal{B}_{\Theta(w)}(\infty)$$:$
\begin{align*}
&v_{\bf i}(\pi^\omega _w(\Xi_{\Theta(w)} ^{\rm up}(b))) = \Omega_{\bf i}(v_{\Theta({\bf i})}(\Xi_{\Theta(w)} ^{\rm up}(b))),\ {\it and}\\
&\tilde{v}_{\bf i}(\pi^\omega _w(\Xi_{\Theta(w)} ^{\rm up}(b))) = \Omega_{\bf i}(\tilde{v}_{\Theta({\bf i})}(\Xi_{\Theta(w)} ^{\rm up}(b))).
\end{align*}
\end{thm}

\begin{proof}
We prove the assertion only for $v_{\bf i}$ and $v_{\Theta({\bf i})}$; the proof of the assertion for $\tilde{v}_{\bf i}$ and $\tilde{v}_{\Theta({\bf i})}$ is similar. We imitate the proof of \cite[Theorem 5.1]{FO}. We write $\Phi_{\Theta({\bf i})}(b) = (a_{1, 1}, \ldots, a_{1, m_{i_1}}, \ldots, a_{r, 1}, \ldots, a_{r, m_{i_r}})$ for $b \in \mathcal{B}_{\Theta(w)}(\infty)$, and proceed by induction on $r = \ell(w)$ and $a_{1, 1} + \cdots + a_{r, m_{i_r}}$. 

We first consider the case $b \in \mathcal{B}_{s_{i_{1, 1}} \cdots s_{i_{1, m_{i_1}}}}(\infty)$, which includes the case $r = 1$. In this case, there exist $a_1, \ldots, a_{m_{i_1}} \in \z_{\ge 0}$ such that $b = \tilde{f}_{i_{1, 1}} ^{a_1} \cdots \tilde{f}_{i_{1, m_{i_1}}} ^{a_{m_{i_1}}} b_\infty$. Then it follows by the definition of $\Phi_{\Theta({\bf i})}$ and assumption (O) in Section 4 that 
\begin{align*}
-v_{\Theta({\bf i})}(\Xi_{\Theta(w)} ^{\rm up}(b)) &= \Phi_{\Theta({\bf i})}(b)\quad({\rm by\ Proposition}\ \ref{string polytopes}\ (1))\\
&= (a_1, \ldots, a_{m_{i_1}}, 0, \ldots, 0).
\end{align*}
Hence we deduce by the definition of $v_{\Theta({\bf i})}$ that $\Xi_{\Theta(w)} ^{\rm up}(b) = c t_{1, 1} ^{a_1} \cdots t_{1, m_{i_1}} ^{a_{m_{i_1}}} + ({\rm other\ terms})$ for some $c \in \c \setminus \{0\}$, where ``other terms'' means a linear combination of monomials of degree $a_1 + \cdots + a_{m_{i_1}}$ that are not equal to $t_{1, 1} ^{a_1} \cdots t_{1, m_{i_1}} ^{a_{m_{i_1}}}$. Here, Lemma \ref{l:positivity lemma} (3) implies that $c \in \mathbb{R}_{>0}$, and that the coefficients of the ``other terms'' are also positive real numbers. Therefore, we see from Lemma \ref{polynomial folding} that $\pi^\omega _w(\Xi_{\Theta(w)} ^{\rm up}(b)) = c^\prime t_1 ^{a_1 + \cdots + a_{m_{i_1}}} + ({\rm other\ terms})$ for some $c^\prime \in \mathbb{R}_{>0}$, where ``other terms'' means a linear combination of monomials in $\mathbb{C}[t_1, \ldots, t_r]$ of degree $a_1 + \cdots + a_{m_{i_1}}$ that are not equal to $t_1 ^{a_1 + \cdots + a_{m_{i_1}}}$. This implies by the definition of $v_{\bf i}$ that 
\begin{align*}
v_{\bf i}(\pi^\omega _w(\Xi_{\Theta(w)} ^{\rm up}(b))) &= -(a_1 + \cdots + a_{m_{i_1}}, 0, \ldots, 0)\\
&= \Omega_{\bf i}(v_{\Theta({\bf i})}(\Xi_{\Theta(w)} ^{\rm up}(b))).
\end{align*}

We next consider the case $r \ge 2$ and $a_{1, 1} = \cdots = a_{1, m_{i_1}} = 0$. In this case, $b$ is an element of $\mathcal{B}_{\Theta(w_{\ge 2})}(\infty)$, where $w_{\ge 2} \coloneqq s_{i_2} \cdots s_{i_r}$. By the definition of $v_{\Theta({\bf i})}$, the equalities $a_{1, 1} = \cdots = a_{1, m_{i_1}} = 0$ imply that $t_{1, 1}, \ldots, t_{1, m_{i_1}}$ do not appear in $\Xi_{\Theta(w)} ^{\rm up}(b)$, and hence that $t_1$ does not appear in $\pi^\omega _w(\Xi_{\Theta(w)} ^{\rm up}(b)) \in \mathbb{C}[t_1, \ldots, t_r]$. From these, we deduce that 
\begin{align*}
v_{\bf i}(\pi^\omega _w(\Xi_{\Theta(w)} ^{\rm up}(b))) &= (0, v_{{\bf i}_{\ge 2}}(\pi^\omega _{w_{\ge 2}}(\Xi_{\Theta(w_{\ge 2})} ^{\rm up}(b))))\\
&= (0, \Omega_{{\bf i}_{\ge 2}}(v_{\Theta({\bf i}_{\ge 2})}(\Xi_{\Theta(w_{\ge 2})} ^{\rm up}(b))))\quad({\rm by\ the\ induction}\ ({\rm on}\ r)\ {\rm hypothesis})\\
&= \Omega_{\bf i}(v_{\Theta({\bf i})}(\Xi_{\Theta(w)} ^{\rm up}(b))),
\end{align*}
where ${\bf i}_{\ge 2} \coloneqq (i_2, \ldots, i_r)$ is a reduced word for $w_{\ge 2}$. 

Finally, consider the case $(a_{1, 1}, \ldots, a_{1, m_{i_1}}) \neq (0, \ldots, 0)$ and $b \notin \mathcal{B}_{s_{i_{1, 1}} \cdots s_{i_{1, m_{i_1}}}}(\infty)$. We set $b_1 \coloneqq \tilde{f}_{i_{1, 1}} ^{a_{1, 1}} \cdots \tilde{f}_{i_{1, m_{i_1}}} ^{a_{1, m_{i_1}}} b_\infty$ and $b_2 \coloneqq \tilde{f}_{i_{2, 1}} ^{a_{2, 1}} \cdots \tilde{f}_{i_{r, m_{i_r}}} ^{a_{r, m_{i_r}}} b_\infty$. Then it follows by the definition of $\Phi_{\Theta({\bf i})}$ that $\Phi_{\Theta({\bf i})}(b_1) = (a_{1, 1}, \ldots, a_{1, m_{i_1}}, 0, \ldots, 0)$ and $\Phi_{\Theta({\bf i})}(b_2) = (0, \ldots, 0, a_{2, 1}, \ldots, a_{r, m_{i_r}})$; here we have used assumption (O) in Section 4. Hence Proposition \ref{string polytopes} (1) implies that 
\begin{align*}
v_{\Theta({\bf i})}(\Xi_{\Theta(w)} ^{\rm up}(b)) &= -(a_{1, 1}, \ldots, a_{r, m_{i_r}})\\
&= -(a_{1, 1}, \ldots, a_{1, m_{i_1}}, 0, \ldots, 0) -(0, \ldots, 0, a_{2, 1}, \ldots, a_{r, m_{i_r}})\\
&= v_{\Theta({\bf i})}(\Xi_{\Theta(w)} ^{\rm up}(b_1)) + v_{\Theta({\bf i})}(\Xi_{\Theta(w)} ^{\rm up}(b_2)).
\end{align*}
Also, we deduce from the induction (on $a_{1, 1} + \cdots + a_{r, m_{i_r}}$) hypothesis that 
\begin{align*}
&\Omega_{\bf i}\left(v_{\Theta({\bf i})}(\Xi_{\Theta(w)} ^{\rm up}(b_1)) + v_{\Theta({\bf i})}(\Xi_{\Theta(w)} ^{\rm up}(b_2))\right)\\
=\ &\Omega_{\bf i}\left(v_{\Theta({\bf i})}(\Xi_{\Theta(w)} ^{\rm up}(b_1))\right) + \Omega_{\bf i}\left(v_{\Theta({\bf i})}(\Xi_{\Theta(w)} ^{\rm up}(b_2))\right)\\
=\ &v_{\bf i}(\pi^\omega _w(\Xi_{\Theta(w)} ^{\rm up}(b_1))) + v_{\bf i}(\pi^\omega _w(\Xi_{\Theta(w)} ^{\rm up}(b_2)))\\
=\ &v_{\bf i}(\pi^\omega _w(\Xi_{\Theta(w)} ^{\rm up}(b_1) \cdot \Xi_{\Theta(w)} ^{\rm up}(b_2)))\\
&({\rm since}\ v_{\bf i}\ {\rm is\ a\ valuation\ and}\ \pi^\omega _w\ {\rm is\ a}\ \c\mathchar`-{\rm algebra\ homomorphism}).
\end{align*}
From these, it follows that 
\begin{align}\label{Key inequalities 1}
v_{\bf i}(\pi^\omega _w(\Xi_{\Theta(w)} ^{\rm up}(b_1) \cdot \Xi_{\Theta(w)} ^{\rm up}(b_2))) = \Omega_{\bf i} (v_{\Theta({\bf i})}(\Xi_{\Theta(w)} ^{\rm up}(b))). 
\end{align}
Here, by Lemma \ref{l:positivity lemma} (2), we have
\begin{align}\label{folded expansion 1}
\Xi_{\Theta(w)} ^{\rm up}(b_1) \cdot \Xi_{\Theta(w)} ^{\rm up}(b_2) = \sum_{b_3 \in \mathcal{B}_{\Theta(w)}(\infty)} C^{(b_3)} _{b_1, b_2} \Xi_{\Theta(w)} ^{\rm up}(b_3) 
\end{align}
for some $C^{(b_3)} _{b_1, b_2} \in \r_{\ge 0}$, $b_3 \in \mathcal{B}_{\Theta(w)}(\infty)$, with $C_{b_1, b_2} ^{(b)} \neq 0$. By applying $\pi^\omega _w$ to (\ref{folded expansion 1}), we obtain
\begin{align}\label{folded expansion 2}
\pi^\omega _w(\Xi_{\Theta(w)} ^{\rm up}(b_1) \cdot \Xi_{\Theta(w)} ^{\rm up}(b_2)) = \sum_{b_3 \in \mathcal{B}_{\Theta(w)}(\infty)} C^{(b_3)} _{b_1, b_2} \pi^\omega _w(\Xi_{\Theta(w)} ^{\rm up}(b_3)). 
\end{align}
Since $C^{(b_3)} _{b_1, b_2} \in \r_{\ge 0}$ for all $b_3 \in \mathcal{B}_{\Theta(w)}(\infty)$, Lemmas \ref{polynomial folding} and \ref{l:positivity lemma} (3) imply that no cancellations of monomials occur in the sum on the right-hand side of (\ref{folded expansion 2}). Therefore, we deduce by the definition of $v_{\bf i}$ that
\begin{align*}
-v_{\bf i}(\pi^\omega _w(\Xi_{\Theta(w)} ^{\rm up}(b_1) \cdot \Xi_{\Theta(w)} ^{\rm up}(b_2))) = \max\{-v_{\bf i}(\pi^\omega _w(\Xi_{\Theta(w)} ^{\rm up}(b_3))) \mid b_3 \in \mathcal{B}_{\Theta(w)}(\infty),\ C^{(b_3)} _{b_1, b_2} \neq 0\},
\end{align*}
where ``max'' means the maximum with respect to the lexicographic order $<$ in Definition \ref{d:valuations}. Since $C_{b_1, b_2} ^{(b)} \neq 0$, we obtain
\begin{align}\label{Key inequalities 2}
-v_{\bf i}(\pi^\omega _w(\Xi_{\Theta(w)} ^{\rm up}(b))) \le -v_{\bf i}(\pi^\omega _w(\Xi_{\Theta(w)} ^{\rm up}(b_1) \cdot \Xi_{\Theta(w)} ^{\rm up}(b_2))).
\end{align}
Now, by the definition of $v_{\Theta({\bf i})}$ together with the equality $-v_{\Theta({\bf i})}(\Xi_{\Theta(w)} ^{\rm up}(b)) = (a_{1, 1}, \ldots, a_{r, m_{i_r}})$, the monomial $t_{1, 1} ^{a_{1, 1}} \cdots t_{r, m_{i_r}} ^{a_{r, m_{i_r}}}$ appears in the polynomial $\Xi_{\Theta(w)} ^{\rm up}(b) \in \c[t_{1, 1}, \ldots, t_{r, m_{i_r}}]$. Since $C_{b_1, b_2} ^{(b)} \neq 0$ and $C^{(b_3)} _{b_1, b_2} \in \r_{\ge 0}$ for all $b_3 \in \mathcal{B}_{\Theta(w)}(\infty)$, we see by Lemmas \ref{polynomial folding} and \ref{l:positivity lemma} (3) that the monomial \[t_1 ^{a_{1, 1} + \cdots + a_{1, m_{i_1}}} \cdots t_r ^{a_{r, 1} + \cdots + a_{r, m_{i_r}}}\] appears in the polynomial $\pi^\omega _w(\Xi_{\Theta(w)} ^{\rm up}(b)) \in \c[t_1, \ldots, t_r]$, which implies that
\begin{equation}\label{Key inequalities 3}
\begin{aligned}
-\Omega_{\bf i} (v_{\Theta({\bf i})}(\Xi_{\Theta(w)} ^{\rm up}(b))) &= (a_{1, 1} + \cdots + a_{1, m_{i_1}}, \ldots, a_{r, 1} + \cdots + a_{r, m_{i_r}})\\ 
&\le -v_{\bf i}(\pi^\omega _w(\Xi_{\Theta(w)} ^{\rm up}(b))).
\end{aligned}
\end{equation}
By combining (\ref{Key inequalities 1}), (\ref{Key inequalities 2}), and (\ref{Key inequalities 3}), we conclude that \[\Omega_{\bf i} (v_{\Theta({\bf i})}(\Xi_{\Theta(w)} ^{\rm up}(b))) = v_{\bf i}(\pi^\omega _w(\Xi_{\Theta(w)} ^{\rm up}(b))) = v_{\bf i}(\pi^\omega _w(\Xi_{\Theta(w)} ^{\rm up}(b_1) \cdot \Xi_{\Theta(w)} ^{\rm up}(b_2))).\] This proves the theorem.
\end{proof}

Denote by $P^\prime \subset (\mathfrak{t}^\ast)^0$ the subgroup generated by $\varpi_i ^\prime \coloneqq \frac{1}{m_i}\sum_{0 \le k < m_i} \varpi_{\omega^k(i)}$, $i \in \breve{I}$. Since the set $\{h_i ^\prime \mid i \in \breve{I}\}$ is regarded as the set of simple coroots of $\mathfrak{g}^\omega$, the subgroup $P^\prime$ is identified with the weight lattice for $\mathfrak{g}^\omega$; in particular, an element $\lambda \in P \cap (\mathfrak{t}^\ast)^0$ gives an integral weight $\hat{\lambda}$ for $\mathfrak{g}^\omega$. Recall that for $w \in \breve{W}$, the Schubert variety $X(w) \subset G^\omega/B^\omega \simeq (G/B)^\omega$ is identified with a Zariski closed subvariety of $X(\Theta(w))$. The inclusion map $X(w) \hookrightarrow X(\Theta(w))$ induces a $B^\omega$-module homomorphism $H^0(X(\Theta(w)), \mathcal{L}_{\lambda}) \rightarrow H^0(X(w), \mathcal{L}_{\hat{\lambda}})$ (denoted also by $\pi^\omega _w$) for $\lambda \in P_+ \cap (\mathfrak{t}^\ast)^0$. Now we define $\mathbb{C}$-linear injective maps $\iota_\lambda \colon H^0(X(\Theta(w)), \mathcal{L}_{\lambda}) \hookrightarrow \mathbb{C}[U^- \cap X(\Theta(w))]$ and $\iota_{\hat{\lambda}} \colon H^0(X(w), \mathcal{L}_{\hat{\lambda}}) \hookrightarrow \mathbb{C}[(U^-)^\omega \cap X(w)]$ as in Lemma \ref{d:iota}. The following is an immediate consequence of the definitions. 

\vspace{2mm}\begin{prop}\label{fixed point lambda proposition}
For $\lambda \in P_+ \cap (\mathfrak{t}^\ast)^0$ and $w \in \breve{W}$, the following diagram is commutative$:$
\begin{align*}
\xymatrix{\c[U^- \cap X(\Theta(w))] \ar[r]^-{\pi^\omega _w} & \c[(U^-)^\omega \cap X(w)]\\
H^0(X(\Theta(w)), \mathcal{L}_{\lambda}) \ar@{^{(}->}[u]^-{\iota_\lambda} \ar[r]^-{\pi^\omega _w} & H^0(X(w), \mathcal{L}_{\hat{\lambda}})\ar@{^{(}->}[u]^-{\iota_{\hat{\lambda}}}.}
\end{align*}
\end{prop}\vspace{2mm}

From this, we obtain the following by Propositions \ref{vanishing} (2), \ref{fixed point lambda proposition}, and Theorem \ref{Key proposition 1}.

\vspace{2mm}\begin{cor}\label{inclusion property}
The following equalities hold$:$
\begin{align*}
&\Omega_{\bf i}(\Delta(X(\Theta(w)), \mathcal{L}_\lambda, v_{\Theta({\bf i})}, \tau_\lambda)) \subset \Delta(X(w), \mathcal{L}_{\hat{\lambda}}, v_{\bf i}, \tau_{\hat{\lambda}}),\ {\it and}\\
&\Omega_{\bf i}(\Delta(X(\Theta(w)), \mathcal{L}_\lambda, \tilde{v}_{\Theta({\bf i})}, \tau_\lambda)) \subset \Delta(X(w), \mathcal{L}_{\hat{\lambda}}, \tilde{v}_{\bf i}, \tau_{\hat{\lambda}}).
\end{align*}
\end{cor}\vspace{2mm}

The following is the main result of this paper.

\vspace{2mm}\begin{thm}\label{surjectivity}
Let ${\bf i} = (i_1, \ldots, i_r) \in \breve{I}^r$ be a reduced word for $w \in \breve{W}$, and $\lambda \in P_+ \cap (\mathfrak{t}^\ast)^0$. Then the maps 
\begin{align*}
&\Omega_{\bf i} \colon \Delta(X(\Theta(w)), \mathcal{L}_\lambda, v_{\Theta({\bf i})}, \tau_\lambda) \rightarrow \Delta(X(w), \mathcal{L}_{\hat{\lambda}}, v_{\bf i}, \tau_{\hat{\lambda}})\ {\it and}\\ 
&\Omega_{\bf i} \colon \Delta(X(\Theta(w)), \mathcal{L}_\lambda, \tilde{v}_{\Theta({\bf i})}, \tau_\lambda) \rightarrow \Delta(X(w), \mathcal{L}_{\hat{\lambda}}, \tilde{v}_{\bf i}, \tau_{\hat{\lambda}})
\end{align*}
are surjective.
\end{thm}\vspace{2mm}

In order to prove this theorem, we consider a pair $((\mathfrak{g}, \omega \colon I \rightarrow I), (\mathfrak{g}^\prime, \omega^\prime \colon I^\prime \rightarrow I^\prime))$ of a simply-laced simple Lie algebra and its Dynkin diagram automorphism. We assume that these satisfy the following conditions:
\begin{enumerate}
\item[${\rm (C)}_1$] the fixed point Lie subalgebra $(\mathfrak{g}^\prime)^{\omega^\prime}$ is isomorphic to the orbit Lie algebra $\breve{\mathfrak{g}}$ associated to $\omega$; this condition implies that the index set $\breve{I}$ for $\breve{\mathfrak{g}}$ is identified with the index set $\breve{I}^\prime\ (= \breve{(I^\prime)})$ for $(\mathfrak{g}^\prime)^{\omega^\prime}$;
\item[${\rm (C)}_2$] if we set $m_i \coloneqq \min\{k \in \mathbb{Z}_{>0} \mid \omega^k(i) = i\}$, $i \in \breve{I}$, and $m_i ^\prime \coloneqq \min\{k \in \mathbb{Z}_{>0} \mid (\omega^\prime)^k(i) = i\}$, $i \in \breve{I}^\prime$, then the product $m_i \cdot m_i ^\prime$ is independent of the choice of $i \in \breve{I} \simeq \breve{I}^\prime$.
\end{enumerate}

\vspace{2mm}\begin{rem}\normalfont
Since the orbit Lie algebra $\breve{\mathfrak g}$ associated to $\omega$ is the (Langlands) dual Lie algebra of the fixed point Lie subalgebra ${\mathfrak g}^\omega$, a pair $((\mathfrak{g}, \omega), (\mathfrak{g}^\prime, \omega^\prime))$ satisfies conditions ${\rm (C)}_1$ and ${\rm (C)}_2$ if and only if a pair $((\mathfrak{g}^\prime, \omega^\prime), (\mathfrak{g}, \omega))$ satisfies these.
\end{rem}\vspace{2mm}

The following three figures give the list of nontrivial pairs satisfying conditions ${\rm (C)}_1$ and ${\rm (C)}_2$:
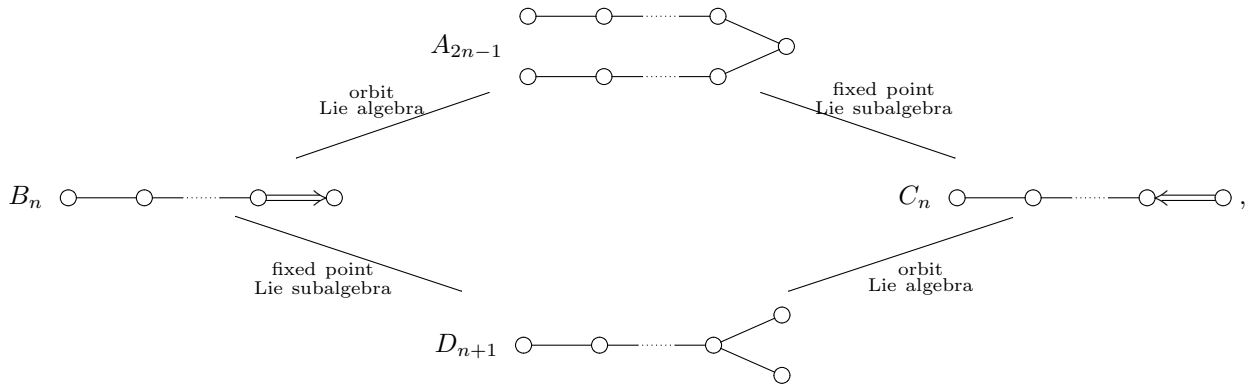
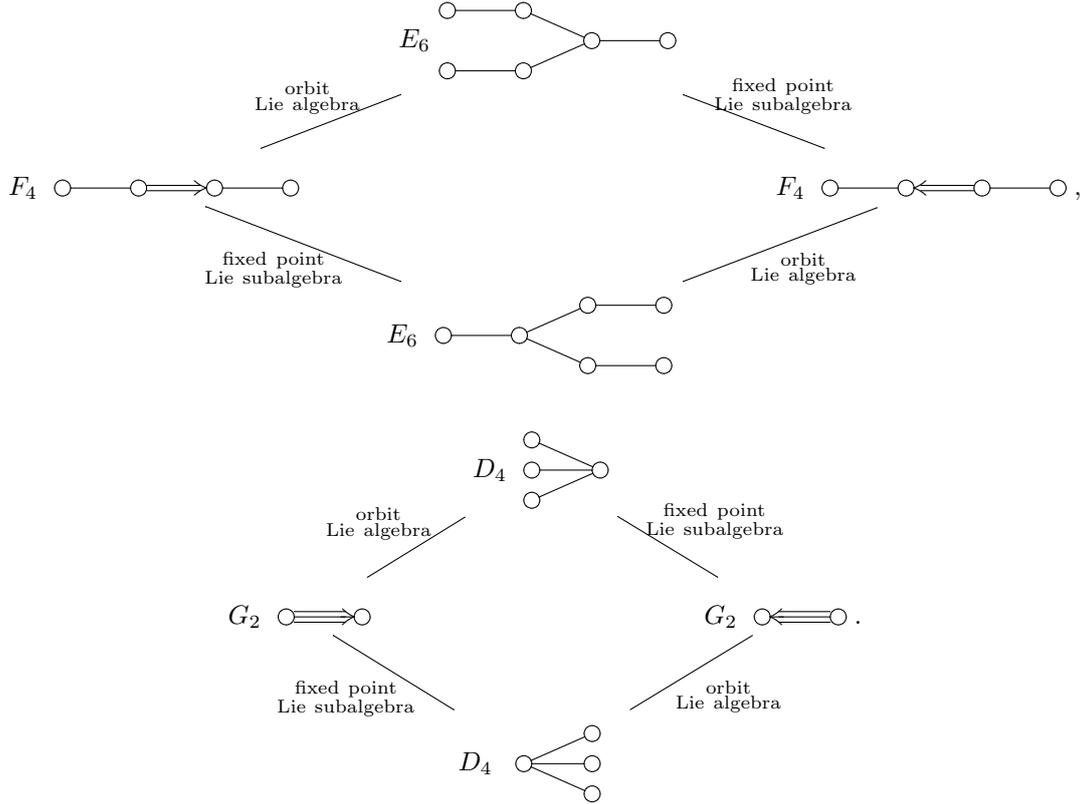
\begin{align*}
\xymatrix{                   
      & A_{2n-1}\ \begin{xy}
\ar@{-} (20,4) *++!D{} *\cir<3pt>{};
(30,4) *++!D!R(0.4){} *\cir<3pt>{}="B",
\ar@{-} "B";(35,4) \ar@{.} (35,4);(40,4)^*!U{}
\ar@{-} (40,4);(45,4) *++!D{} *\cir<3pt>{}="C"
\ar@{-} "C";(54,0) *++!L{} *\cir<3pt>{}="D"
\ar@{-} "D";(45,-4) *++!D{} *\cir<3pt>{}="E"
\ar@{-} "E";(40,-4) \ar@{.} (40,-4);(35,-4)^*!U{}
\ar@{-} (35,-4);(30,-4) *++!D{} *\cir<3pt>{}="F"
\ar@{-} "F";(20,-4) *++!D{} *\cir<3pt>{}
\end{xy} \ar@{-}[rd]^-{\substack{{\rm fixed\ point}\\{\rm Lie\ subalgebra}}} &   \\
B_n\ \begin{xy}
\ar@{-} (50,0) *++!D{} *\cir<3pt>{};
(60,0) *++!D{} *\cir<3pt>{}="C"
\ar@{-} "C";(65,0) \ar@{.} (65,0);(70,0)^*!U{}
\ar@{-} (70,0);(75,0) *++!D{} *\cir<3pt>{}="D"
\ar@{=>} "D";(85,0) *++!D{} *\cir<3pt>{}="E"
\end{xy} \ar@{-}[ru]^-{\substack{{\rm orbit}\\{\rm Lie\ algebra}}} &                            & \ar@{-}[ld]^-{\substack{{\rm orbit}\\{\rm Lie\ algebra}}} C_n\ \begin{xy}
\ar@{-} (50,0) *++!D{} *\cir<3pt>{};
(60,0) *++!D{} *\cir<3pt>{}="C"
\ar@{-} "C";(65,0) \ar@{.} (65,0);(70,0)^*!U{}
\ar@{-} (70,0);(75,0) *++!D{} *\cir<3pt>{}="D"
\ar@{<=} "D";(85,0) *++!D{} *\cir<3pt>{}="E"
\end{xy},\\
                         & \ar@{-}[lu]^-{\substack{{\rm fixed\ point}\\{\rm Lie\ subalgebra}}} D_{n+1}\ \begin{xy}
\ar@{-} (20,0) *++!D{} *\cir<3pt>{};
(30,0) *++!D!R(0.4){} *\cir<3pt>{}="B"
\ar@{-} "B";(35,0) \ar@{.} (35,0);(40,0)^*!U{}
\ar@{-} (40,0);(45,0) *++!D{} *\cir<3pt>{}="C"
\ar@{-} "C";(54,4) *++!L{} *\cir<3pt>{}
\ar@{-} "C";(54,-4) *++!L{} *\cir<3pt>{},
\end{xy}&   
} 
\end{align*}
\begin{align*}
\xymatrix{                   
      & E_6\ \begin{xy}
\ar@{-} (40,0) *++!D{} *\cir<3pt>{};
(30,0) *++!D!R(0.4){} *\cir<3pt>{}="C",
\ar@{-} "C";(21,4) *++!L{} *\cir<3pt>{}="D"
\ar@{-} "C";(21,-4) *++!L{} *\cir<3pt>{}="E"
\ar@{-} "D";(11,4) *++!L{} *\cir<3pt>{}
\ar@{-} "E";(11,-4) *++!L{} *\cir<3pt>{}
\end{xy} \ar@{-}[rd]^-{\substack{{\rm fixed\ point}\\{\rm Lie\ subalgebra}}} &   \\
F_4\ \begin{xy}
\ar@{-} (50,0) *++!D{} *\cir<3pt>{};
(60,0) *++!D{} *\cir<3pt>{}="B"
\ar@{=>} "B";(70,0) *++!D{} *\cir<3pt>{}="C"
\ar@{-} "C";(80,0) *++!D{} *\cir<3pt>{}
\end{xy} \ar@{-}[ru]^-{\substack{{\rm orbit}\\{\rm Lie\ algebra}}} &                            & \ar@{-}[ld]^-{\substack{{\rm orbit}\\{\rm Lie\ algebra}}} F_4\ \begin{xy}
\ar@{-} (50,0) *++!D{} *\cir<3pt>{};
(60,0) *++!D{} *\cir<3pt>{}="B"
\ar@{<=} "B";(70,0) *++!D{} *\cir<3pt>{}="C"
\ar@{-} "C";(80,0) *++!D{} *\cir<3pt>{}
\end{xy},\\
                         & \ar@{-}[lu]^-{\substack{{\rm fixed\ point}\\{\rm Lie\ subalgebra}}} E_6\ \begin{xy}
\ar@{-} (20,0) *++!D{} *\cir<3pt>{};
(30,0) *++!D!R(0.4){} *\cir<3pt>{}="C",
\ar@{-} "C";(39,4) *++!L{} *\cir<3pt>{}="D"
\ar@{-} "C";(39,-4) *++!L{} *\cir<3pt>{}="E",
\ar@{-} "D";(49,4) *++!L{} *\cir<3pt>{}
\ar@{-} "E";(49,-4) *++!L{} *\cir<3pt>{}
\end{xy}&   
} 
\end{align*}
\begin{align*}
\xymatrix{                   
      & D_4\ \begin{xy}
\ar@{-} (21,0) *++!D{} *\cir<3pt>{};
(30,0) *++!D!R(0.4){} *\cir<3pt>{}="C",
\ar@{-} "C";(21,4) *++!L{} *\cir<3pt>{}
\ar@{-} "C";(21,-4) *++!L{} *\cir<3pt>{}
\end{xy} \ar@{-}[rd]^-{\substack{{\rm fixed\ point}\\{\rm Lie\ subalgebra}}} &   \\
G_2\ \begin{xy}
\ar@3{->} *++!D{} *\cir<3pt>{};
(10,0) *++!D{} *\cir<3pt>{}
\end{xy} \ar@{-}[ru]^-{\substack{{\rm orbit}\\{\rm Lie\ algebra}}} &                            & \ar@{-}[ld]^-{\substack{{\rm orbit}\\{\rm Lie\ algebra}}} G_2\ \begin{xy}
\ar@3{<-} *++!D{} *\cir<3pt>{};
(10,0) *++!D{} *\cir<3pt>{}
\end{xy}.\\
                         & \ar@{-}[lu]^-{\substack{{\rm fixed\ point}\\{\rm Lie\ subalgebra}}} D_4\ \begin{xy}
\ar@{-} (29,0) *++!D{} *\cir<3pt>{};
(20,0) *++!D!R(0.4){} *\cir<3pt>{}="C",
\ar@{-} "C";(29,4) *++!L{} *\cir<3pt>{}
\ar@{-} "C";(29,-4) *++!L{} *\cir<3pt>{}
\end{xy}&   
} 
\end{align*}
By this list and Table \ref{table1} in Section 4, we obtain the following.

\vspace{2mm}\begin{prop}
For a simply-laced simple Lie algebra $\mathfrak{g}$ with a Dynkin diagram automorphism $\omega$, there exists a simply-laced simple Lie algebra $\mathfrak{g}^\prime$ with a Dynkin diagram automorphism $\omega^\prime$ such that $((\mathfrak{g}, \omega), (\mathfrak{g}^\prime, \omega^\prime))$ satisfies conditions ${\rm (C)}_1$ and ${\rm (C)}_2$.
\end{prop}\vspace{2mm}

For simplicity, we consider only the pair $(A_{2n-1}, D_{n+1})$; we note that all the arguments below carry over to the other pairs. Denote the Weyl group of type $A_{2n-1}$ by $W^{A_{2n-1}}$, the Schubert variety of type $A_{2n-1}$ by $X^{A_{2n-1}}(w)$, and so on. We identify $\breve{I} \coloneqq \{1, \ldots, n\}$ with the set of vertices of the Dynkin diagram of type $B_n$, and also with that of type $C_n$ as follows:
\begin{align*}
B_n\ \begin{xy}
\ar@{-} (50,0) *++!D{1} *\cir<3pt>{};
(60,0) *++!D{2} *\cir<3pt>{}="C"
\ar@{-} "C";(65,0) \ar@{.} (65,0);(70,0)^*!U{}
\ar@{-} (70,0);(75,0) *++!D{n-1} *\cir<3pt>{}="D"
\ar@{=>} "D";(85,0) *++!D{n} *\cir<3pt>{}="E"
\end{xy}\hspace{-1mm},\\ C_n\ \begin{xy}
\ar@{-} (50,0) *++!D{1} *\cir<3pt>{};
(60,0) *++!D{2} *\cir<3pt>{}="C"
\ar@{-} "C";(65,0) \ar@{.} (65,0);(70,0)^*!U{}
\ar@{-} (70,0);(75,0) *++!D{n-1} *\cir<3pt>{}="D"
\ar@{<=} "D";(85,0) *++!D{n} *\cir<3pt>{}="E"
\end{xy}\hspace{-1mm}.
\end{align*}
Note that the Weyl group $W^{B_n}$ is isomorphic to the Weyl group $W^{C_n}$. As we have seen in Section 4, the Weyl group $W^{B_n}\ (\simeq W^{C_n})$ is regarded as a specific subgroup of $W^{A_{2n-1}}$ (resp., of $W^{D_{n+1}}$); let $\Theta \colon W^{B_n} \hookrightarrow W^{A_{2n-1}}$ (resp., $\Theta' \colon W^{B_n} \hookrightarrow W^{D_{n+1}}$) be the inclusion map. Take a reduced word ${\bf i} = (i_1, \ldots, i_r) \in \breve{I}^r$ for $w \in W^{B_n} \simeq W^{C_n}$. The reduced word ${\bf i}$ induces a reduced word $\Theta({\bf i})$ (resp., $\Theta'({\bf i})$) for $\Theta(w)$ (resp., for $\Theta'(w)$); see Section 4. By Corollary \ref{Key corollary 1} and Theorem \ref{Key proposition 1}, we obtain the following diagrams; we denote the map $\Omega_{\bf i} \colon \Phi_{\Theta({\bf i})} (\mathcal{B}^{A_{2n-1}} _{\Theta(w)}(\infty)) \rightarrow \Phi_{\bf i} (\mathcal{B}^{C_n} _w(\infty))$ by $\Omega_{\bf i} ^{A, C}$, the map $\Upsilon_{\bf i} \colon \Phi_{\bf i} (\mathcal{B}^{B_n} _w(\infty)) \rightarrow \Phi_{\Theta({\bf i})} (\mathcal{B}^{A_{2n-1}} _{\Theta(w)}(\infty))$ by $\Upsilon_{\bf i} ^{B, A}$, and so on.
\begin{align*}
\xymatrix{                   
      & \Phi_{\Theta({\bf i})} (\mathcal{B}^{A_{2n-1}} _{\Theta(w)}(\infty)) \ar@{->}[rd]^-{\Omega_{\bf i} ^{A, C}} &   \\
\Phi_{\bf i} (\mathcal{B}^{B_n} _w(\infty)) \ar@{^{(}->}[ru]^-{\Upsilon_{\bf i} ^{B, A}} &                            & \ar@{^{(}->}[ld]^-{\Upsilon_{\bf i} ^{C, D}} \Phi_{\bf i} (\mathcal{B}^{C_n} _w(\infty)),\\
                         & \ar@{->}[lu]^-{\Omega_{\bf i} ^{D, B}} \Phi_{\Theta'({\bf i})} (\mathcal{B}^{D_{n+1}} _{\Theta'(w)}(\infty))&   
} 
\end{align*}
\begin{align*}
\xymatrix{                   
      & \Psi_{\Theta({\bf i})} (\mathcal{B}^{A_{2n-1}} _{\Theta(w)}(\infty)) \ar@{->}[rd]^-{\Omega_{\bf i} ^{A, C}} &   \\
\Psi_{\bf i} (\mathcal{B}^{B_n} _w(\infty)) \ar@{^{(}->}[ru]^-{\Upsilon_{\bf i} ^{B, A}} &                            & \ar@{^{(}->}[ld]^-{\Upsilon_{\bf i} ^{C, D}} \Psi_{\bf i} (\mathcal{B}^{C_n} _w(\infty)).\\
                         & \ar@{->}[lu]^-{\Omega_{\bf i} ^{D, B}} \Psi_{\Theta'({\bf i})} (\mathcal{B}^{D_{n+1}} _{\Theta'(w)}(\infty))&   
} 
\end{align*}

\begin{proof}[Proof of Theorem \ref{surjectivity}]
We give a proof of the assertion only for the map \[\Omega_{\bf i} ^{A, C} \colon \Delta(X^{A_{2n-1}} (\Theta(w)), \mathcal{L}_{\lambda}, v_{\Theta({\bf i})}, \tau_{\lambda}) \rightarrow \Delta(X^{C_n}(w), \mathcal{L}_{\hat{\lambda}}, v_{\bf i}, \tau_{\hat{\lambda}});\] the proofs for the other cases are similar. Because
\begin{align*}
&\Delta(X^{A_{2n-1}}(\Theta(w)), \mathcal{L}_{2\lambda}, v_{\Theta({\bf i})}, \tau_{2\lambda}) = 2\Delta(X^{A_{2n-1}}(\Theta(w)), \mathcal{L}_{\lambda}, v_{\Theta({\bf i})}, \tau_{\lambda})\ {\rm and}\\
&\Delta(X^{C_n}(w), \mathcal{L}_{2\hat{\lambda}}, v_{\bf i}, \tau_{2\hat{\lambda}}) = 2\Delta(X^{C_n}(w), \mathcal{L}_{\hat{\lambda}}, v_{\bf i}, \tau_{\hat{\lambda}}), 
\end{align*}
it suffices to prove that the map 
\begin{equation}\label{goal of surjectivity}
\begin{aligned}
\Omega_{\bf i} ^{A, C} \colon \Delta(X^{A_{2n-1}}(\Theta(w)), \mathcal{L}_{2\lambda}, v_{\Theta({\bf i})}, \tau_{2\lambda}) \rightarrow \Delta(X^{C_n}(w), \mathcal{L}_{2\hat{\lambda}}, v_{\bf i}, \tau_{2\hat{\lambda}}) 
\end{aligned}
\end{equation}
is surjective. By the definitions of $\Omega_{\bf i}$ and $\Upsilon_{\bf i}$, we see that $\Omega_{\bf i} ^{A, C} \circ \Upsilon_{\bf i} ^{B, A} (a_1, \ldots, a_r) = (a_1 ^\prime, \ldots, a_r ^\prime)$ and $\Omega_{\bf i} ^{D, B} \circ \Upsilon_{\bf i} ^{C, D} (a_1, \ldots, a_r) = (a'' _1, \ldots, a'' _r)$ for $(a_1, \ldots, a_r) \in \r^r$, where 
\begin{equation}\label{explicit description of composite}
\begin{aligned}
&a_k ^\prime \coloneqq \begin{cases}
2a_k\quad(i_k = 1, \ldots, n-1),\\
a_k\quad\ (i_k = n),
\end{cases}\\
&a'' _k \coloneqq \begin{cases}
a_k\quad\ (i_k = 1, \ldots, n-1),\\
2a_k\quad(i_k = n)
\end{cases}
\end{aligned}
\end{equation}
for $k = 1, \ldots, r$. From these, it follows that the composite map $\Omega_{\bf i} ^{A, C} \circ \Upsilon_{\bf i} ^{B, A} \circ \Omega_{\bf i} ^{D, B} \circ \Upsilon_{\bf i} ^{C, D}$ is identical to $2 \cdot {\rm id}_{\r^r}$. This implies that the map \[\Omega_{\bf i} ^{A, C} \circ \Upsilon_{\bf i} ^{B, A} \circ \Omega_{\bf i} ^{D, B} \circ \Upsilon_{\bf i} ^{C, D} \colon \Delta(X^{C_n} (w), \mathcal{L}_{\hat{\lambda}}, v_{\bf i}, \tau_{\hat{\lambda}}) \rightarrow \Delta(X^{C_n} (w), \mathcal{L}_{2\hat{\lambda}}, v_{\bf i}, \tau_{2\hat{\lambda}})\] doubles each of the coordinates, and hence is surjective. Therefore, the map \eqref{goal of surjectivity} is also surjective. This proves the theorem.
\end{proof}

\begin{ex}\normalfont
Consider the case $n = 2$:
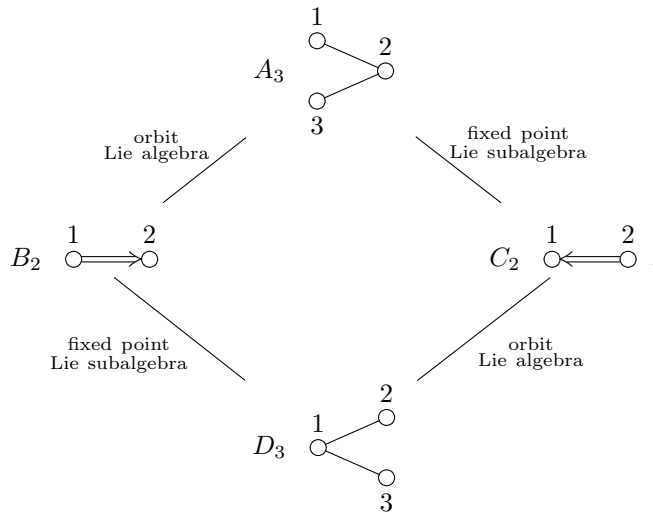
\begin{align*}
\xymatrix{                   
      & A_3\ \begin{xy}
\ar@{-} (45,4) *++!D{1} *\cir<3pt>{};(54,0) *++!D{2} *\cir<3pt>{}="D"
\ar@{-} "D";(45,-4) *++!U{3} *\cir<3pt>{}="E"
\end{xy} \ar@{-}[rd]^-{\substack{{\rm fixed\ point}\\{\rm Lie\ subalgebra}}} &   \\
B_2\ \begin{xy}
\ar@{=>} (75,0) *++!D{1} *\cir<3pt>{}="C";(85,0) *++!D{2} *\cir<3pt>{}
\end{xy} \ar@{-}[ru]^-{\substack{{\rm orbit}\\{\rm Lie\ algebra}}} &                            & \ar@{-}[ld]^-{\substack{{\rm orbit}\\{\rm Lie\ algebra}}} C_2\ \begin{xy}
\ar@{<=} (75,0) *++!D{1} *\cir<3pt>{}="C";(85,0) *++!D{2} *\cir<3pt>{}
\end{xy}.\\
                         & \ar@{-}[lu]^-{\substack{{\rm fixed\ point}\\{\rm Lie\ subalgebra}}} D_3\ \begin{xy}
\ar@{-} (45,0) *++!D{1} *\cir<3pt>{}="C";(54,4) *++!D{2} *\cir<3pt>{}
\ar@{-} "C";(54,-4) *++!U{3} *\cir<3pt>{},
\end{xy}&   
} 
\end{align*}
Set ${\bf i} \coloneqq (1, 2, 1) \in \breve{I}^3$; this is a reduced word for $w \coloneqq s_1 s_2 s_1 \in W^{B_2} \simeq W^{C_2}$. By the definitions of $\Theta$ and $\Theta'$, we have $\Theta({\bf i}) = (1, 3, 2, 1, 3)$ and $\Theta'({\bf i}) = (1, 2, 3, 1)$. Then, it follows from \cite[Sect.\ 1]{Lit} that 
\begin{align*}
&\Phi_{\Theta({\bf i})}(\mathcal{B}_{\Theta(w)} ^{A_3}(\infty)) = \{(a_1, \ldots, a_5) \in \z_{\ge 0} ^5 \mid a_3 \ge a_4,\ a_3 \ge a_5\},\\ 
&\Phi_{\bf i}(\mathcal{B}_w ^{B_2}(\infty)) = \{(a_1, a_2, a_3) \in \z_{\ge 0} ^3 \mid a_2 \ge a_3\},\\
&\Phi_{\bf i}(\mathcal{B}_w ^{C_2}(\infty)) = \{(a_1, a_2, a_3) \in \z_{\ge 0} ^3 \mid 2 a_2 \ge a_3\},\\ 
&\Phi_{\Theta'({\bf i})}(\mathcal{B}_{\Theta'(w)} ^{D_3}(\infty)) = \{(a_1, \ldots, a_4) \in \z_{\ge 0} ^4 \mid a_2 + a_3 \ge a_4\}.
\end{align*}
In addition, the maps $\Omega_{\bf i} ^{A, C} \colon \r^5 \twoheadrightarrow \r^3$, $\Upsilon_{\bf i} ^{B, A} \colon \r^3 \hookrightarrow \r^5$, $\Omega_{\bf i} ^{D, B} \colon \r^4 \twoheadrightarrow \r^3$, and $\Upsilon_{\bf i} ^{C, D} \colon \r^3 \hookrightarrow \r^4$ are given by
\begin{align*}
&\Omega_{\bf i} ^{A, C}(a_1, \ldots, a_5) \coloneqq (a_1 + a_2, a_3, a_4 + a_5),\ \Upsilon_{\bf i} ^{B, A}(a_1, a_2, a_3) \coloneqq (a_1, a_1, a_2, a_3, a_3),\\ 
&\Omega_{\bf i} ^{D, B}(a_1, \ldots, a_4) \coloneqq (a_1, a_2 + a_3, a_4),\ \Upsilon_{\bf i} ^{C, D}(a_1, a_2, a_3) \coloneqq (a_1, a_2, a_2, a_3).
\end{align*}
Through the map $\Omega_{\bf i} ^{A, C}$, the conditions $a_3 \ge a_4,\ a_3 \ge a_5$ for $\Phi_{\Theta({\bf i})}(\mathcal{B}_{\Theta(w)} ^{A_3}(\infty))$ correspond to the condition $2 a_2 \ge a_3$ for $\Phi_{\bf i}(\mathcal{B}_w ^{C_2}(\infty))$; hence we see that $\Omega_{\bf i} ^{A, C}(\Phi_{\Theta({\bf i})}(\mathcal{B}_{\Theta(w)} ^{A_3}(\infty))) = \Phi_{\bf i}(\mathcal{B}_w ^{C_2}(\infty))$. Similarly, we observe that the following equalities hold: 
\begin{align*}
&\Omega_{\bf i} ^{D, B}(\Phi_{\Theta'({\bf i})}(\mathcal{B}_{\Theta'(w)} ^{D_3}(\infty))) = \Phi_{\bf i}(\mathcal{B}_w ^{B_2}(\infty)),\\
&\Upsilon_{\bf i} ^{B, A}(\Phi_{\bf i}(\mathcal{B}_w ^{B_2}(\infty))) = \{(a_1, \ldots, a_5) \in \Phi_{\Theta({\bf i})}(\mathcal{B}_{\Theta(w)} ^{A_3}(\infty)) \mid a_1 = a_2,\ a_4 = a_5\},\\
&\Upsilon_{\bf i} ^{C, D}(\Phi_{\bf i}(\mathcal{B}_w ^{C_2}(\infty))) = \{(a_1, \ldots, a_4) \in \Phi_{\Theta'({\bf i})}(\mathcal{B}_{\Theta'(w)} ^{D_3}(\infty)) \mid a_2 = a_3\}.
\end{align*}
Take $\lambda \in P_+ ^{A_3} \cap (\mathfrak{t}^\ast)^0$ and set $\lambda_i \coloneqq \langle\lambda, h_i ^{A_3}\rangle$ for $i = 1, 2, 3$. The condition $\lambda \in (\mathfrak{t}^\ast)^0$ implies that $\lambda_1 = \lambda_3$. By the definition of $\hat{\lambda}$, it follows that $\langle\hat{\lambda}, h_1 ^{C_2}\rangle = 2 \lambda_1 = 2 \lambda_3$ and $\langle\hat{\lambda}, h_2 ^{C_2}\rangle = \lambda_2$. Therefore, we see from Proposition \ref{string polytopes} (2) and \cite[Sect.\ 1]{Lit} that $-\Delta(X^{A_{2n-1}} (\Theta(w)), \mathcal{L}_{\lambda}, v_{\Theta({\bf i})}, \tau_{\lambda})$ (resp., $-\Delta(X^{C_n}(w), \mathcal{L}_{\hat{\lambda}}, v_{\bf i}, \tau_{\hat{\lambda}})$) is given by the following conditions:
\begin{align*}
&(a_1, \ldots, a_5) \in \r_{\ge 0} ^5,\ a_3 \ge a_4,\ a_3 \ge a_5,\ a_5 \le \lambda_1,\ a_4 \le \lambda_1,\\ 
&a_3 \le \lambda_2 + a_4 + a_5,\ a_2 \le \lambda_1 + a_3 - 2a_5,\ a_1 \le \lambda_1 + a_3 -2a_4\\
({\rm resp.},\ &(a_1, a_2, a_3) \in \r_{\ge 0} ^3,\ 2 a_2 \ge a_3,\ a_3 \le 2 \lambda_1,\ a_2 \le \lambda_2 + a_3,\ a_1 \le 2 \lambda_1 + 2a_2 -2a_3).
\end{align*}
Hence it follows that \[\Omega_{\bf i} ^{A, C}(\Delta(X^{A_{2n-1}} (\Theta(w)), \mathcal{L}_{\lambda}, v_{\Theta({\bf i})}, \tau_{\lambda})) = \Delta(X^{C_n}(w), \mathcal{L}_{\hat{\lambda}}, v_{\bf i}, \tau_{\hat{\lambda}}).\]
\end{ex}\vspace{2mm}

\section{Relation with similarity of crystal bases}

In this section, we study the relation of the folding procedure discussed in Sections 4, 5 with a similarity of crystal bases. 

First we review (a variant of) a similarity property of crystal bases, following \cite[Sect.\ 5]{Kas5}. Let $\mathfrak{g}, I, P, \{\alpha_i, h_i \mid i \in I\}$ be as in Section 2, and take $m_i \in \z_{>0}$ for every $i \in I$. We set $\tilde{\alpha}_i \coloneqq m_i \alpha_i$, $\tilde{h}_i \coloneqq \frac{1}{m_i} h_i$ for $i \in I$, and denote by $\widetilde{P} \subset P$ the set of those $\lambda \in P$ such that $\langle\lambda, \tilde{h}_i\rangle \in \mathbb{Z}$ for all $i \in I$. We impose the following condition on $\{m_i \mid i \in I\}$: 
\[\tilde{\alpha}_i \in \widetilde{P}\ {\rm for\ all}\ i \in I.\]
Then, it is easily seen that the matrix $(\langle \tilde{\alpha}_j, \tilde{h}_i\rangle)_{i, j \in I}$ is an indecomposable Cartan matrix of finite type. Let $\mathfrak{g}^\prime$ be the corresponding simple Lie algebra. Note that the set $\widetilde{P}$ is identified with the weight lattice for $\mathfrak{g}^\prime$. Let us write $\mathcal{B}(\infty)$ for $\mathfrak{g}$ as $\mathcal{B}^{\mathfrak{g}}(\infty)$, $\mathcal{B}(\lambda)$ for ${\mathfrak g}$ as $\mathcal{B}^{\mathfrak{g}}(\lambda)$, and so on.

\vspace{2mm}\begin{prop}[{see the proof of \cite[Theorem 5.1]{Kas5}}]\label{similarity existence}
There exists a unique map $S_\infty \colon \mathcal{B}^{\mathfrak{g}^\prime}(\infty) \rightarrow \mathcal{B}^{\mathfrak{g}}(\infty)$ satisfying the following conditions$:$
\begin{enumerate}
\item[{\rm (i)}] $S_\infty (b_\infty ^{\mathfrak{g}^\prime}) = b_\infty ^{\mathfrak{g}}$,
\item[{\rm (ii)}] $S_\infty (\tilde{X}_i b) = \tilde{X}_i ^{m_i} S_\infty (b)$ for all $i \in I$, $b \in \mathcal{B}^{\mathfrak{g}^\prime}(\infty)$, and $X \in \{e, f\}$, where $S_\infty(0) \coloneqq 0$.
\end{enumerate}
\end{prop}\vspace{2mm}

If $\mathfrak{g}$ is of type $B_n$ and $(m_1, \ldots, m_{n-1}, m_n) = (1, \ldots, 1, 2)$, then $\mathfrak{g}^\prime$ is the simple Lie algebra of type $C_n$. Conversely, if $\mathfrak{g}$ is of type $C_n$ and $(m_1, \ldots, m_{n-1}, m_n) = (2, \ldots, 2, 1)$, then $\mathfrak{g}^\prime$ is the simple Lie algebra of type $B_n$. Hence we obtain the following.

\vspace{2mm}\begin{cor}\label{similarity existence B,C}
The following hold.
\begin{enumerate}
\item[{\rm (1)}] There exists a unique map $S_\infty ^{B, C} \colon \mathcal{B}^{B_n}(\infty) \rightarrow \mathcal{B}^{C_n}(\infty)$ satisfying the following conditions$:$
\begin{enumerate}
\item[{\rm (i)}] $S_\infty ^{B, C}(b_\infty ^{B_n}) = b_\infty ^{C_n}$,
\item[{\rm (ii)}] $S_\infty ^{B, C}(\tilde{X}_i b) = \begin{cases}
\widetilde{X}_i ^2 S_\infty ^{B, C}(b)\quad(i = 1, \ldots, n-1),\\
\widetilde{X}_n S_\infty ^{B, C}(b)\quad(i = n)
\end{cases}$\\
for all $b \in \mathcal{B}^{B_n}(\infty)$ and $X \in \{e, f\}$, where $S_\infty ^{B, C}(0) \coloneqq 0$.
\end{enumerate}
\item[{\rm (2)}] There exists a unique map $S_\infty ^{C, B} \colon \mathcal{B}^{C_n}(\infty) \rightarrow \mathcal{B}^{B_n}(\infty)$ satisfying the following conditions$:$
\begin{enumerate}
\item[{\rm (i)}] $S_\infty ^{C, B}(b_\infty ^{C_n}) = b_\infty ^{B_n}$,
\item[{\rm (ii)}] $S_\infty ^{C, B}(\tilde{X}_{i}b) = \begin{cases}
\widetilde{X}_i S_\infty ^{C, B}(b)\quad(i = 1, \ldots, n-1),\\
\widetilde{X}_n ^2 S_\infty ^{C, B}(b)\quad(i = n)
\end{cases}$\\
for all $b \in \mathcal{B}^{C_n}(\infty)$ and $X \in \{e, f\}$, where $S_\infty ^{C, B}(0) \coloneqq 0$.
\end{enumerate}
\end{enumerate}
\end{cor}\vspace{2mm}

It is easily seen that the composite map $S_\infty ^{C, B} \circ S_\infty ^{B, C}$ is identical to the map $S_2 ^{B} \colon \mathcal{B}^{B_n}(\infty) \rightarrow \mathcal{B}^{B_n}(\infty)$ given by the following conditions:
\begin{enumerate}
\item[{\rm (i)}] $S_2 ^{B}(b_\infty ^{B_n}) = b_\infty ^{B_n}$,
\item[{\rm (ii)}] $S_2 ^{B}(\tilde{X}_i b) = \tilde{X}_i ^2 S_2 ^{B}(b)$ for all $i \in I$, $b \in \mathcal{B}^{B_n}(\infty)$, and $X \in \{e, f\}$, where $S_2 ^B(0) \coloneqq 0$,
\item[{\rm (iii)}] $\varepsilon_i(S_2 ^{B}(b)) = 2 \varepsilon_i(b)$ and $\varphi_i(S_2 ^{B}(b)) = 2 \varphi_i(b)$ for all $i \in I$ and $b \in \mathcal{B}^{B_n}(\infty);$
\end{enumerate}
see also \cite[Theorem 3.1]{Kas5}. Similar result holds for the composite map $S_\infty ^{B, C} \circ S_\infty ^{C, B} \colon \mathcal{B}^{C_n}(\infty) \rightarrow \mathcal{B}^{C_n}(\infty)$. Recall that the Weyl group of type $B_n$ is isomorphic to that of type $C_n$. By conditions (i) and (ii) in Corollary \ref{similarity existence B,C} (1) (resp., (2)), it follows that \[S_\infty ^{B, C}(\mathcal{B}_w ^{B_n}(\infty)) \subset \mathcal{B}_w ^{C_n}(\infty)\ {\rm (resp.,}\ S_\infty ^{C, B}(\mathcal{B}_w ^{C_n}(\infty)) \subset \mathcal{B}_w ^{B_n}(\infty))\] for all $w \in W^{B_n} \simeq W^{C_n}$.

\vspace{2mm}\begin{prop}\label{p:last proposition}
Let ${\bf i} = (i_1, \ldots, i_r) \in I^r$ be a reduced word for $w \in W^{B_n} \simeq W^{C_n}$. Then the following equalities hold for all $b \in \mathcal{B}^{B_n} _w(\infty)$ and $b^\prime \in \mathcal{B}^{C_n}_w(\infty)$$:$
\begin{align*}
&\Phi_{\bf i}(S_\infty ^{B, C}(b)) = \Omega_{\bf i} ^{A, C} \circ \Upsilon_{\bf i} ^{B, A}(\Phi_{\bf i}(b)),\ \Phi_{\bf i}(S_\infty ^{C, B}(b^\prime)) = \Omega_{\bf i} ^{D, B} \circ \Upsilon_{\bf i} ^{C, D}(\Phi_{\bf i} (b^\prime)),\ {\it and}\\
&\Psi_{\bf i}(S_\infty ^{B, C}(b)) = \Omega_{\bf i} ^{A, C} \circ \Upsilon_{\bf i} ^{B, A} (\Psi_{\bf i} (b)),\ \Psi_{\bf i}(S_\infty ^{C, B}(b^\prime)) = \Omega_{\bf i} ^{D, B} \circ \Upsilon_{\bf i} ^{C, D} (\Psi_{\bf i} (b^\prime)).
\end{align*}
\end{prop}

\begin{proof}
We prove the assertion only for $S_\infty ^{B, C}$; the proof of the assertion for $S_\infty ^{C, B}$ is similar. By equation (\ref{explicit description of composite}) in the proof of Theorem \ref{surjectivity}, it suffices to prove that 
\begin{align*}
&\varepsilon_i(S_\infty ^{B, C}(b)) = \begin{cases}
2\varepsilon_i (b)\quad(i = 1, \ldots, n-1),\\
\varepsilon_i (b)\quad\ (i = n),\end{cases}\\
&\varepsilon_i(S_\infty ^{B, C}(b)^\ast) = \begin{cases}
2\varepsilon_i (b^\ast)\quad(i = 1, \ldots, n-1),\\
\varepsilon_i (b^\ast)\quad\ (i = n)\end{cases} 
\end{align*}
for all $b \in \mathcal{B}^{B_n}(\infty)$. The assertion for $\varepsilon_i(S_\infty ^{B, C}(b)^\ast)$ follows immediately from the proof of \cite[Theorem 5.1]{Kas5}. We will prove the assertion for $\varepsilon_i(S_\infty ^{B, C}(b))$. If $i = n$, then this is obvious by condition (ii) in Corollary \ref{similarity existence B,C} (1). For $i = 1, \ldots, n-1$, we see by condition (ii) in Corollary \ref{similarity existence B,C} (1) that \[\tilde{e}_i ^{2 \varepsilon_i(b)}S_\infty ^{B, C}(b) = S_\infty ^{B, C}(\tilde{e}_i ^{\varepsilon_i(b)}b) \neq 0.\] Suppose, for a contradiction, that $\tilde{e}_i ^{2 \varepsilon_i(b)+1}S_\infty ^{B, C}(b) \neq 0$. Then we have 
\begin{align*}
\tilde{e}_i ^{2 \varepsilon_i(b)+1}S_2 ^{B}(b) &= \tilde{e}_i ^{2 \varepsilon_i(b)+1}S_\infty ^{C, B} \circ S_\infty ^{B, C}(b)\\
&= S_\infty ^{C, B}(\tilde{e}_i ^{2 \varepsilon_i(b)+1}S_\infty ^{B, C}(b))\quad({\rm by\ condition\ (ii)\ in\ Corollary}\ \ref{similarity existence B,C}\ {\rm (2)})\\
&\neq 0,
\end{align*}
which contradicts condition (iii) for $S^{B} _2$ above. Therefore, the equality $\tilde{e}_i ^{2 \varepsilon_i(b)+1}S_\infty ^{B, C}(b) = 0$ holds. From these, we deduce that $\varepsilon_i(S_\infty ^{B, C}(b)) = 2\varepsilon_i(b)$. This proves the proposition.
\end{proof}

\begin{rem}\normalfont
Proposition \ref{p:last proposition} is naturally extended to an arbitrary pair $((\mathfrak{g}, \omega), (\mathfrak{g}^\prime, \omega^\prime))$ satisfying conditions ${\rm (C)}_1$ and ${\rm (C)}_2$ in Section 5.
\end{rem}\vspace{2mm}

\appendix
\section{Case of affine Lie algebras}

Our arguments in this paper are naturally extended to symmetrizable Kac-Moody algebras. The following figures give the list of nontrivial pairs of automorphisms of simply-laced affine Dynkin diagrams satisfying conditions ${\rm (C)}_1$ and ${\rm (C)}_2$ in Section 5; we have used Kac's notation.
\begin{align*}
\xymatrix{                   
      & A_{2n-1} ^{(1)} \ar@{-}[rd]^-{\substack{{\rm fixed\ point}\\{\rm Lie\ subalgebra}}} &   &  &  & D_{n+1} ^{(1)} \ar@{-}[rd]^-{\substack{{\rm fixed\ point}\\{\rm Lie\ subalgebra}}} &   \\
D_{n+1} ^{(2)} \ar@{-}[ru]^-{\substack{{\rm orbit}\\{\rm Lie\ algebra}}} &                            & \ar@{-}[ld]^-{\substack{{\rm orbit}\\{\rm Lie\ algebra}}} C_n ^{(1)}, & & A_{2n-1} ^{(2)} \ar@{-}[ru]^-{\substack{{\rm orbit}\\{\rm Lie\ algebra}}} &                            & \ar@{-}[ld]^-{\substack{{\rm orbit}\\{\rm Lie\ algebra}}} B_n ^{(1)},\\
        & \ar@{-}[lu]^-{\substack{{\rm fixed\ point}\\{\rm Lie\ subalgebra}}} D_{n+2} ^{(1)}&   &                        &  & \ar@{-}[lu]^-{\substack{{\rm fixed\ point}\\{\rm Lie\ subalgebra}}} D_{2n} ^{(1)}&   
} 
\end{align*}
\begin{align*}
\xymatrix{                   
      & D_{2n} ^{(1)} \ar@{-}[rd]^-{\substack{{\rm fixed\ point}\\{\rm Lie\ subalgebra}}} &   &  &  & A_3 ^{(1)} \ar@{-}[rd]^-{\substack{{\rm fixed\ point}\\{\rm Lie\ subalgebra}}} &   \\
A_{2n-2} ^{(2)} \ar@{-}[ru]^-{\substack{{\rm orbit}\\{\rm Lie\ algebra}}} &                            & \ar@{-}[ld]^-{\substack{{\rm orbit}\\{\rm Lie\ algebra}}} A_{2n-2} ^{(2)}, & & A_1 ^{(1)} \ar@{-}[ru]^-{\substack{{\rm orbit}\\{\rm Lie\ algebra}}} &                            & \ar@{-}[ld]^-{\substack{{\rm orbit}\\{\rm Lie\ algebra}}} A_1 ^{(1)},\\
        & \ar@{-}[lu]^-{\substack{{\rm fixed\ point}\\{\rm Lie\ subalgebra}}} D_{2n} ^{(1)}&   &                        &  & \ar@{-}[lu]^-{\substack{{\rm fixed\ point}\\{\rm Lie\ subalgebra}}} A_3 ^{(1)}&   
} 
\end{align*}
\begin{align*}
\xymatrix{                   
      & E_6 ^{(1)} \ar@{-}[rd]^-{\substack{{\rm fixed\ point}\\{\rm Lie\ subalgebra}}} &   &  &  & D_4 ^{(1)} \ar@{-}[rd]^-{\substack{{\rm fixed\ point}\\{\rm Lie\ subalgebra}}} &   \\
E_6 ^{(2)} \ar@{-}[ru]^-{\substack{{\rm orbit}\\{\rm Lie\ algebra}}} &                            & \ar@{-}[ld]^-{\substack{{\rm orbit}\\{\rm Lie\ algebra}}} F_4 ^{(1)}, & & D_4 ^{(3)} \ar@{-}[ru]^-{\substack{{\rm orbit}\\{\rm Lie\ algebra}}} &                            & \ar@{-}[ld]^-{\substack{{\rm orbit}\\{\rm Lie\ algebra}}} G_2 ^{(1)}.\\
        & \ar@{-}[lu]^-{\substack{{\rm fixed\ point}\\{\rm Lie\ subalgebra}}} E_7 ^{(1)}&   &                        &  & \ar@{-}[lu]^-{\substack{{\rm fixed\ point}\\{\rm Lie\ subalgebra}}} E_6 ^{(1)}&   
} 
\end{align*}

\end{document}